\documentclass[a4paper,letterpaper,leqno]{amsart}

\usepackage{amsmath}
\usepackage{amsthm}
\usepackage{amssymb}
\usepackage[english]{babel}
\usepackage{lscape}
\usepackage[latin1]{inputenc}
\usepackage{enumerate}
\usepackage{fullpage}

\newtheorem{theorem}{Theorem}[section]
\newtheorem{conjecture}[theorem]{Conjecture}
\newtheorem{remark}[theorem]{Remark}
\newtheorem{corollary}[theorem]{Corollary}
\newtheorem{lemma}[theorem]{Lemma}
\newtheorem{proposition}[theorem]{Proposition}
\newtheorem{definition}[theorem]{Definition}
\newtheorem{example}[theorem]{Example}

\newcommand{\C}{\mathbb{C}}
\newcommand{\Z}{\mathbb{Z}}

\newcommand{\Q}{\mathbb{Q}}
\newcommand{\F}{\mathbb{F}}
\newcommand{\Qbar}{\overline{\mathbb{Q}}}
\newcommand{\ST}{\operatorname{ST}}
\newcommand{\AST}{\operatorname{AST}}
\newcommand{\Hg}{\operatorname{Hg}}
\newcommand{\Cc}{\mathcal{C}}
\newcommand{\p}{\mathfrak{p}}
\newcommand{\Pp}{\mathfrak{P}}
\newcommand{\Frob}{\operatorname{Frob}}

\newcommand{\Ker}{\operatorname{Ker}}
\newcommand{\ord}{\operatorname{ord}}
\newcommand{\id}{\operatorname{id}}
\newcommand{\Lcm}{\operatorname{lcm}}
\newcommand{\Jac}{\operatorname{Jac}}
\newcommand{\Unitary}{\operatorname{U}}
\newcommand{\USp}{\operatorname{USp}}
\newcommand{\diag}{\operatorname{diag}}
\newcommand{\End}{\operatorname{End}}
\newcommand{\cM}{\mathcal M}

\newcommand{\Tr}{\operatorname{Tr}}

\newcommand{\rk}{\operatorname{rk}}
\newcommand{\M}{\operatorname{M}}
\newcommand{\Ind}{\operatorname{Ind}}
\newcommand{\Lef}{\operatorname{L}}
\newcommand{\TL}{\operatorname{TL}}
\newcommand{\Sp}{\operatorname{Sp}}

\newcommand{\Aut}{\operatorname{Aut}}
\newcommand{\Gal}{\operatorname{Gal}}
\newcommand{\GL}{\operatorname{GL}}

\begin{document}

\title{Frobenius distribution for quotients\\ of Fermat curves of prime exponent}
\date{\today}

\author{Francesc Fit\'e}
\address{Francesc Fit\'e\\
Institut f\"ur Experimentelle Mathematik/Fakult\"at f\"ur Mathematik, Universit\"at Duisburg-Essen, D-45127 Essen\\
email: francesc.fite@gmail.com}
\author{Josep Gonz\'alez}
\address{Josep Gonz\'alez\\
Departament de Matem\`atica Aplicada IV, Universitat Polit\`ecnica de Catalunya, Av. V\'ictor Balaguer s/n., E-08800 Vilanova i la Geltr\'u\\
email: josepg@ma4.upc.edu}
\author{Joan-C. Lario}
\address{Joan-C. Lario\\
Departament de Matem\`atica Aplicada II, Universitat Polit\`ecnica de Catalunya, Edifici Omega-Campus Nord, Jordi Girona 1-3, E-08034 Barcelona\\
email: joan.carles.lario@upc.edu}

\begin{abstract}
Let $\Cc$ denote the Fermat curve over $\Q$ of prime exponent $\ell$. The Jacobian $\Jac(\Cc)$ of~$\Cc$ splits over $\Q$ as the product of Jacobians $\Jac(\Cc_k)$, $1\leq k\leq \ell-2$, where $\Cc_k$ are curves obtained as quotients of $\Cc$ by certain subgroups of automorphisms of $\Cc$. It is well known that $\Jac(\Cc_k)$ is the power of an absolutely simple abelian variety $B_k$ with complex multiplication. We call degenerate those pairs $(\ell,k)$ for which $B_k$ has degenerate CM type. For a non-degenerate pair $(\ell,k)$, we compute the Sato-Tate group of $\Jac(\Cc_k)$, prove the generalized Sato-Tate Conjecture for it, and give an explicit method to compute the moments and measures of the involved distributions. Regardless of $(\ell,k)$ being degenerate or not, we also obtain Frobenius equidistribution results for primes of certain residue degrees in the $\ell$-th cyclotomic field. Key to our results is a detailed study of the rank of certain generalized Demjanenko matrices.
\end{abstract}


\maketitle

\dedicatory{Dedicated to Josep Gran\'e on the occasion of his $70$th birthday}
\tableofcontents

\section{Introduction}

Both from the theoretical and the computational points of view, the problem of determining Frobenius distributions of low genus curves has attracted a growing interest in the past years (see for example \cite{KS08}, \cite{Ser12}, \cite{FKRS12}, and \cite{FS13}). 

In this paper, we consider this problem for a family of curves of arbitrary high genus and which have simple Jacobian in many cases. More concretely, for a prime $\ell$ and an integer $1\leq k \leq \ell-2$, we are concerned with the limiting distribution of the normalized Euler local factors $L_p(\Cc_k,T/\sqrt p)$ attached to the curves $\Cc_k$ defined by the affine equation
$$
v^{\ell}=u(u+1)^{\ell-k-1}\,.
$$
The curves $\Cc_k$ have genus $\frac{\ell-1}{2}$ and may be obtained as quotients of the Fermat curve $\Cc:y^\ell=x^\ell+1$ by certain subgroups of automorphisms of $\Cc$. One can in fact show that the Jacobian of $\Cc$ decomposes up to isogeny over $\Q$ as the product
$$
\Jac(\Cc)\sim_\Q\prod_{k=1}^{\ell-2}\Jac(\Cc_k)\,.
$$
It is also well known that the $L$-function of $\Jac(\Cc_k)$ can be written in terms of Hecke $L$-functions attached to Jacobi sums and that $\Jac(\Cc_k)$ is the power of an absolutely simple abelian variety $B_k$ with complex multiplication, say of dimension $r_k$. This is all recalled in the preliminary \S\ref{section: simple factors}.

We now procced to describe the three main results of the paper. We say that a pair $(\ell,k)$ is non-degenerate if the dimension of the Hodge group $\Hg(B_k)$ is maximal (that is, equal to $r_k$). This is equivalent to saying that the CM-type of $B_k$ is non-degenerate (in the sense of Kubota) or that the determinant of the Demjanenko matrix $D_k$ does not vanish. The $r_k\times r_k$ matrix $D_k$ can be elementarily constructed, and has been extensively studied in the literature (see for example \cite{Haz90}, \cite{SS95}, or \cite{Doh94}). 

The generalized Sato-Tate Conjecture for $\Jac(\Cc_k)$ predicts the existence of a compact real Lie group $\ST(\Jac(\Cc_k))\subseteq \USp(\ell-1)$, the Sato-Tate group, that determines the limiting distribution of the normalized Euler local factors $L_p(\Cc_k,T/\sqrt p)$. The main result in \S\ref{section: ST group} is the computation of $\ST(\Jac(\Cc_k))$ along with a proof of the equidistribution predicted by the generalized Sato-Tate Conjecture (see Proposition \ref{proposition: ST group}, Conjecture \ref{conjecture: ST}, and  Theorem \ref{theorem: ST Q}).

\begin{theorem}
If $(\ell,k)$ is a non-degenerate pair, the generalized Sato-Tate Conjecture holds for $\Jac(\Cc_k)$.
\end{theorem}

The connected component of $\ST(\Jac(\Cc_k))$ is a product of a certain number of copies of the unitary group $\Unitary(1)$. This is easy to deduce from the well-known  structure of $\Hg(\Jac(\Cc_k))$. The main novelty concerning the computation of $\ST(\Jac(\Cc_k))$ is the description of its group of components. Abstractly, the group of components is simply isomorphic to $(\Z/\ell\Z)^*$. However, if one seeks to provide an explicit description of the limiting distribution of $L_p(\Cc_k,T/\sqrt p)$, one needs to supply an explicit embedding of $\ST(\Jac(\Cc_k))$ inside $\USp(\ell-1)$.

Following Serre's general strategy, our proof of equidistribution is based on the nonvanishing of certain $L$-functions attached to the irreducible nontrivial representations of $\ST(\Jac(\Cc_k))$. Then, we apply a theorem of Hecke on the holomorphicity and nonvanishing for $\Re(s)\geq 1$ of the $L$-functions attached to unitarized nontrivial Hecke characters. It should be mentioned that it is precisely the nonvanishing of $\det(D_k)$ that ensures that the Hecke character involved in the proof, which is constructed by means of Jacobi sums, is nontrivial. A proof of equidistribution for abelian varieties with complex multiplication in general is certainly well known to the experts (see Johansson \cite{Joh13}). However, in our particular example, once one has an explicit description of $\ST(\Jac(\Cc_k))$, one can show equidistribution in an elementary and very explicit way, that we illustrate in Theorem \ref{theorem: ST Q}.

Naturally, we call degenerate the pairs $(\ell,k)$ for which $\rk(D_k)<r_k$. Building on several known results (see \cite{Kub65}, \cite{Gre80}, \cite{Rib80}), we prove in \S\ref{section: degenerate primes} the following theorem (see Theorem \ref{theorem: det deg}).

\begin{theorem}\label{theorem: carac}
A pair $(\ell,k)$ is degenerate if and only if the three following conditions hold:
\begin{enumerate}[i)]
\item $k$ is not a primitive cubic root of unity modulo $\ell$;
\item $\ord(-k^2-k)$ and $\ord(k)$ are odd, where $\ord$ means the order in $(\Z/\ell\Z)^*$;
\item $v_3( \ord(k)) > v_3(\ord(k^2+k))$, where $v_3$ denotes the $3$-adic valuation.
\end{enumerate}
In this case, $\dim(\Hg(\Jac(\Cc_k)))=\rk(D_k)= \frac{\ell-1}{2}\left(1-\frac{2}{N_k}\right)$, where $N_k:=\Lcm (\ord(-k^2-k),\ord(k))$. 
\end{theorem}
The previous result has several consequences. As an example, one can deduce that the rank of~$D_k$ is ``asymptotically non-degenerate", that is,
$$
\lim_{\begin{array}{c}\ell \rightarrow \infty\\ 1\leq k\leq\ell-2 \end{array}}\frac{\rk(D_k)}{r_k}=1\,.
$$
However, our interest in the previous theorem is motivated by the fact that it constitutes the fundamental technical result for the discussion in \S\ref{section: explicit distributions}. In that section, we are concerned with limiting distributions when restricting to primes of a given residue degree $f$ in $\Q(\zeta_\ell)$. 
With this in mind, for each divisor $f$ of $\ell-1$, we define a matrix $D_{k,f}$, which may be seen as a generalization of $D_k$, and say that $f$ is a \emph{non-$k$-degenerate residue degree} if $\det(D_{k,f})\not=0$. We then achieve the following concrete characterization. A divisor $f$ of~$\ell-1$ is non-$k$-degenerate if and only if it is odd and either:
\begin{enumerate}[i)] 
\item $(\ell,k)$ is non-degenerate; or
\item $(\ell,k)$ is degenerate and $f\in \mathcal F_0 \cup \mathcal F_1$.
\end{enumerate} 
Here, $\mathcal F_0$ (resp. $\mathcal F_1$) is the set of odd divisors of $\ell-1$ such that $v_3(f)=v_3(N_k)-1$ and $f$ is a multiple of $N_k/3$ (resp. such that $v_3(f)\geq v_3(N_k)$). We say that $f$ is $k$-degenerate otherwise. Note that when $(\ell,k)$ is degenerate, there are still non-$k$-degenerate divisors $f$ of $\ell-1$.
In \S\ref{section: explicit distributions}, we develope a method to compute the limiting distribution of $L_p(\Cc_k,T/\sqrt p)$ when restricting to primes $p$ of any fixed \emph{non-$k$-degenerate} residue degree~$f$ in~$\Q(\zeta_\ell)$. This method is based on a detailed analysis of the local factors of $\Jac(\Cc_k)$ and works independently of the construction of $\ST(\Jac(\Cc_k))$ (which we recall that we are only able to achieve for non-degenerate pairs $(\ell,k)$). 

The rank of $D_{k,f}$ depends on the size $n_{k,f}$ of a certain subgroup $W_{k,f}$ of $(\Z/\ell\Z)^*$. A comprehensive description of the subgroup $W_{k,f}$ will be given in \S\ref{section: degenerate primes}. For the purpose of stating our last main result, it will suffice for the moment to mention that
$$
n_{k,f}=
\begin{cases}
3f & \text{if $k$ is a primitive cubic root of unity and }v_3(f)=0\,,\\
3f & \text{if $(\ell,k)$ is degenerate and $f\in \mathcal F_0$}\,,\\
f & \text{otherwise.}
\end{cases}
$$
The following theorem is a combination of Corollary \ref{corollary: GR} and Theorem \ref{theorem: equidist f}.
\begin{theorem}\label{theorem: intro} Let $p\not= \ell$ be a prime of residue degree $f$ in $\Q(\zeta_\ell)$. Then
$$
L_p(\Cc_k,T/\sqrt p)=
\begin{cases} 
\displaystyle{(1+T^f)^{\frac{\ell-1}{f}}} & \text{if $f$ is even,}\\[4pt]
\displaystyle{\prod_{i=1}^{r_{k,f}}(1+s_i(p)T^f+T^{2f})^{\frac{n_{k,f}}{f}}} & \text{if $f$ is non-$k$-degenerate,}
\end{cases}
$$ 
where $s_i(p)\in[-2,2]$ and $r_{k,f}=\frac{\ell-1}{2n_{k,f}}$.  Moreover, in the case that $f$ is non-$k$-degenerate, the sequence $\{(s_1(p),\dots,s_{r_{k,f}}(p)) \}_p$, where $p\not=\ell$ runs over the set of primes of residue degree $f$ in $\Q(\zeta_\ell)$, is equidistributed over $[-2,2]^{r_{k,f}} /\mathfrak{S}_{r_{k,f}}$ with respect to the measure
$
\prod_{i=1}^{r_{k,f}}\frac{1}{\pi}\frac{dx_i}{\sqrt{4-x_i^2}}$, where $\mathfrak{S}_{r_{k,f}}$ denotes the symmetric group on $r_{k,f}$ letters.
\end{theorem}

The case of even $f$ is easy and well-known (see \cite[Lemma 1.1]{GR78}). The proof for the case of non-$k$-degenerate $f$ relies again on the result of Hecke mentioned above. Now, the nontriviality of the Hecke character appearing in the core of the proof is ensured by the nonvanishing of the determinant of the matrix $D_{k,f}$. In \S\ref{section: numerical data}, examples and numerical data are shown. Tables~\ref{table: moments} and~\ref{table: statistical moments} compare the theoretical moments with the moment statistics up to $2^{27}$ for the first trace of $L_p(\Cc_k,T/\sqrt p)$ for a few choices of $(\ell,k)$.
\vskip 0.3truecm
\textbf{Notation and conventions.} Throughout the article, we will write $\zeta_\ell$ for a primitive $\ell$th root of unity, $F:=\Q(\zeta_\ell)$, $G:=(\Z/\ell\Z)^ *$, and $g$ for a generator of $G$. Also, for $a\in G$ or $a$ a rational number with denominator coprime to $\ell$, the symbol $\langle a \rangle\in \Z$ will denote the unique representative of~$a$ modulo~$\ell$ between~$0$ and~$\ell-1$ (to avoid any potential confusion: in any case should $\langle a \rangle$ be thought as the subgroup of of $G$ generated by $a$). We will use left exponential notation for Galois actions. We will identify $G$ and $\Gal(F/\Q)$ via the isomorphism
\begin{equation}\label{equation: iso groups}
G\rightarrow\Gal(F/\Q)\,,\qquad t\mapsto \sigma_t\,,\qquad\text{where ${}^{\sigma_t}(\zeta_\ell):=\zeta_\ell^t$}\,.  
\end{equation}
Any number field will $K$ be assumed to belong to a fixed algebraic closure $\overline \Q$ of $\Q$, and we will write $G_K:=\Gal(\overline\Q/K)$ for its absolute Galois group. We will refer to the prime ideals of the ring of integers~$\mathcal O_K$ of a number field~$K$, simply as primes of $K$. For an algebraic variety~$X$ defined over a number field~$K$ and an extension $L/K$, we will denote by $X_L$ the base change of~$X$ to~$L$.
For an abelian variety~$A$ over~$K$ and a prime ideal~$\p$ of~$K$ of good reduction for~$A$, we will let $L_\p(A,T)=\prod_{i=1}^{2\dim(A)}(1-\alpha_iT)$ denote the local factor of $A$ at $\p$, that is the polynomial with the defining property that for each positive integer $n$
$$
|A(\F_{q^n})| = \prod_{i=1}^{2\dim(A)}(1-\alpha_i^n)\,,
$$
where $q=N{\p}$ is the norm of $\p$.

\textbf{Acknowledgements.} Thanks to Santiago Molina for helpful discussions, to Anna Somoza and Joan S\'anchez for computer assitance in the elaboration of Table \ref{table: statistical moments}, and to the referee for helpful suggestions. The authors were partially supported by MECD project MTM2012-34611. The first author received finantial support from the German Research Council via CRC 701.

\section{Preliminaries}\label{section: simple factors}

Fix a prime $\ell\geq 3$. Let us denote by $ \Cc$  the Fermat curve over $\Q$ defined as the
projective closure of the affine curve
$
y^{\ell} =x^{\ell} +1 \,.
$
It is well-known that  the curve $ \Cc$ has genus $\binom{\ell-1}{2}$, good reduction at all primes $p\neq \ell$, and no singular points. Moreover,
the set
$$
\omega_{i,j}:= \frac{x^{j-1}}{y^i}d x\,,
$$
for $2\leq i \leq \ell-1$ and $1\leq j\leq i-1$ is a basis of the $\Q$-vector space of regular differentials $\Omega^1_{\Cc/\Q}$.
In this section we will recall  results concerning the decomposition of the Jacobian of $\Cc$ over $\Q$.  We will particularize a result of \cite{KR78} to the case of prime exponent~$\ell$ that completely caracterizes the absolutely simple factors of this decomposition. 
Then we will introduce the Hecke characters that describe the $L$-functions attached to these simple factors.

\subsection{Decomposing the Jacobian of a Fermat curve}

\begin{proposition} For every integer $k$ such that $1\leq k\le \ell-2$ we denote by $\Cc_k$ the normalization of the projective
closure
of the affine curve given by the equation
\begin{equation}\label{afi}  v^{\ell}=u(u+1)^{\ell-k-1}\,.\end{equation}
\begin{enumerate}[i)]
\item  The morphism $\pi_k: \Cc\rightarrow \Cc_k$  defined
by the assignment
$(x,y)\mapsto(u,v)=(x^{\ell}, x\,y^{\ell-k-1})$ has degree $\ell$.
\item Let $\mathcal A_k$ be the $\Gal(F/\Q)$-stable subgroup of automorphisms of $\mathcal C$ generated by $\gamma_k$, where $\gamma_k$ is defined by the assignment $(x,y)\mapsto (x\,\zeta_\ell^{k+1},y\,\zeta_\ell)$. The curve $\Cc_k$ is the quotient curve of $\Cc$ by $\mathcal A_k$ and its genus is $\frac{\ell-1}{2}$.
\item $\Jac(\Cc)$ is isogenous over $\Q$ to the abelian variety
$\prod_{k=1}^{\ell-2}\Jac(\Cc_k)$.
\end{enumerate}
\end{proposition}

\begin{proof}
Since $\gamma_k$ has order $\ell$ and  $\pi_k\circ \gamma_k=\pi_k$,
 it follows  that $\Cc_k\simeq \Cc/\mathcal A_k$. Due to the fact that
 $\mathcal A_k$ is $\Gal (F/\Q)$-stable,
 the isomorphism  between both curves is defined over $\Q$. The automorphism  $\gamma_k$ has no fixed points and, thus, $\pi_k$
 is unramified. By applying Hurwitz's formula, we deduce that the genus of $\Cc_k$ is equal to $(\ell-1)/2$, yielding  $ii)$.

It can be checked that $\pi_k^*(\Omega^1_{\Cc_k/\Q})=(\Omega^1_{\Cc/\Q})^{\mathcal A_k}$ is the space generated by the differentials $\omega_{i,j}$, where
\begin{equation}\label{equation: differentials}
  2\leq i \leq \ell-1\,, 1\leq j\leq i-1\,,  (k+1)\cdot j\equiv i\pmod{\ell}\,
 \end{equation}
and hence we obtain that
 $$
\Omega^1_{\Cc/\Q}=\bigoplus_{k=2}^{\ell-1}\pi_k^*(\Omega^1_{\Cc_k/\Q})\,,
 $$
and $iii)$ follows.
\end{proof}

For $r,s\geq 1$ with $r+s\leq \ell-1$, define the set $H_{r,s}$ and the group $V_{r,s}$ as
$$
H_{r,s}:=\{ j \in G\,|\, \langle rj\rangle + \langle sj\rangle < \ell \}\,,\qquad
 V_{r,s}:=\{ w\in G \,|\, wH_{r,s}=H_{r,s}\}\,.
$$

\begin{remark}\label{remark: equality H}
Observe that $H_{r,s}=tH_{\langle tr\rangle,\langle ts\rangle}$ for any $t\in G$.
\end{remark}

\begin{theorem}[see Theorem 1 of \cite{KR78},\cite{ST61}]\label{theorem: KR} For $1\leq k\leq \ell-2$, we have
\begin{enumerate}[i)]
\item $\Jac(\Cc_k)\sim_{F} B_k^{| V_{k,1}|}$, where $B_k$ is a simple abelian variety defined over $F$ of dimension $\frac{\ell-1}{2| V_{k,1}|}$ with CM by the fixed subfield $F^ {V_{k,1}}$ and CM type equal to $H_{k,1}/V_{k,1}$;
\item $\Jac(\Cc_k)\sim_{F} \Jac(\Cc_{k'})$ if and only if we have an equality of sets $H_{k,1}=H_{k',1}$.
\item $H_{r,s}=H_{r',s'}$ if and only if $\{r,s,\ell-r-s\}=\{\langle t\rangle ,\langle tr'\rangle ,\langle -t(r'+s')\rangle\}$ for some $t\in G$.
\end{enumerate}
\end{theorem}

\begin{remark} The previous theorem is stated in \cite{KR78} in terms of abelian varieties $A_{k,1}$ defined by certain lattices $L_{k,1}$. To see that $A_{k,1}$ and $\Jac(\Cc_k)$ coincide we refer the reader to the Appendix of Rohrlich in  \cite{Gro78}.
\end{remark}

Let $S$ and $T$ be the permutations of the set of indices  $\{1,\cdots,\ell-2\}\subseteq \Z$ defined by
 $T(k):=\left\langle \frac{-k}{k+1} \right\rangle$ and $S(T):=\left\langle \frac{1}{k} \right\rangle$, and let $\cM$ be the group of permutations
 generated by $S$ and $T$, which is isomorphic to the dihedral group of six elements. The orbit of $k$ under the action of $\cM$ is
 $$
\cM(k):= \left\{ M(k)\,|\, M\in \cM\right\}=\left\{k,\left\langle\frac{-1}{k+1}\right\rangle,\left\langle \frac{-k-1}{k}\right\rangle,\left\langle \frac{-k}{k+1}\right\rangle,\left\langle -k-1\right\rangle,\left\langle \frac{1}{k}\right\rangle\right\}\,.
 $$
For $\ell>3$, this set has generically six elements, except when either $k\in \{\ell-2,1,(\ell-1)/2\}$, in which case $\cM (k)=\{\ell-2,1,(\ell-1)/2\}$, or $k$ is a primitive cubic root of unity modulo $\ell$, in which case $\cM(k)=\{k, \ell-k-1\}$.

\begin{lemma}\label{lemma: splitting} For $1\leq k,k'\leq\ell-2$, the following statements are equivalent:
\begin{enumerate}[i)]
\item $ \Cc_k\simeq_\Q  \Cc_{k'}$.
\item $ \Jac(\Cc_k)\sim_F  \Jac(\Cc_{k'})$.
\item $k'\in \cM(k)$.
\end{enumerate}
In particular, for $\ell>3$, there are exactly $(\ell+5)/6$ or $ (\ell+1)/6$ isogeny classes among the jacobians of the $\ell-2$ curves $\Cc_k$ depending on whether $\ell\equiv 1\pmod 3$ or not.
\end{lemma}

\begin{proof}
It is obvious that $i)$ implies $ii)$.
It is a straighforward computation to see that $\{1,k,\ell-k-1\}=\{\langle t\rangle ,\langle tk'\rangle ,\langle -t(k'+1)\rangle\}$ for some $t\in G$ if and only if $k'\in \cM(k)$, and then use $ii)$ and $iii)$ of Theorem \ref{theorem: KR}. To see that $iii)$ implies $i)$, let $\lambda,\mu\in \Aut_{\Q} (\Cc)$ be the automorphisms defined by de assignments
$$
\lambda\colon\,\, (x,y)\mapsto (-y,-x)\,,\quad \mu\colon\,\, (x,y)\mapsto ( 1/x,y/x)\,.$$
Since
$$
\gamma_{\langle -k/(k+1)\rangle} ^{-(k+1)}\circ \lambda\circ \gamma_k=\lambda\quad \text{and}\quad
\gamma_{\langle 1/k\rangle} ^{k}\circ \mu\circ \gamma_k=\mu\,,
$$
it follows that $\lambda$ (resp. $\mu$) induces an isomorphism between $\Cc_{\langle -k/(k+1)\rangle}$ (resp.
 $\Cc_{\langle 1/k\rangle}$)  and $\Cc_k$ defined over $\Q$. Thus, $\Cc_{M(k)}$  and $\Cc_k$ are isomorphic over $\Q$ for all $M\in\cM(k)$.
\end{proof}

For $1\leq k \leq \ell-2$, set $M_{k}:=H_{k,1}$ and $W_k:=V_{k,1}$ in order to simplify notation. Let $n_k$ denote the cardinality of $W_k$. Note that we have
\begin{equation}\label{equation: MW}
M_k=\{j\in G\,|\, \langle kj\rangle+\langle j\rangle < \ell \}=\{j\in G\,|\, \langle j\rangle < \langle (k+1)j\rangle \}\,,\qquad
 W_k=\{ w\in G\,|\, wM_k=M_k\}\,.
\end{equation}

\begin{remark}\label{remark: basis differentials}
Observe that the previous description of $M_k$, together with the conditions in (\ref{equation: differentials}), shows that 
$$
\omega_j:=\frac{x^{\langle j\rangle -1}}{y^{\langle (k+1)j\rangle}}dx\qquad \text{with}\qquad j\in M_k
$$
is a basis of $\pi_k^*\Big(\Omega^1_{\Cc_k/\Q}\Big)$.
\end{remark}

\begin{lemma}[see Theorem 2 of \cite{KR78}]\label{lemma: card W}
For $1\leq k \leq \ell-2$, we have that $n_k$ is $1$ or $3$. Moreover, $n_k=3$ if and only if $k$ is a primitive cubic root of unity modulo $\ell$. In this case, we have that
$$
W_k=\{1,k,-k-1\}\subseteq G\,.
$$
\end{lemma}

\begin{proof}
We will show that for $1\leq k \leq \ell-2$, the cardinality of $V_{k,1}$ is $1$ or $3$. Suppose that $w\not = 1$ lies in $V_{k,1}$, that is, $w H_{k,1}=H_{k,1}$. Note that by Remark \ref{remark: equality H} this means that $H_{w k,w}=H_{k,1}$. We will show that both $w$ and $k$ are primitive cubic roots of unity modulo~$\ell$. By Theorem \ref{theorem: KR} part iii), we have
$$
\{1,k,\ell-k-1\}=\{\langle w\rangle ,\langle w k\rangle,\langle -w(k+1)\rangle\}\,.
$$
Therefore either $w \equiv k\pmod \ell$ or $w\equiv -k-1\pmod \ell$. Suppose first that $w\equiv k \pmod \ell$. Then
$$
\{1,w,\ell-w-1\}=\{\langle w\rangle ,\langle w^2\rangle,\langle -w^2-w\rangle\}\,.
$$
Then either $1\equiv w^2 \,\pmod \ell$ or $1\equiv -w-w^2 \,\pmod \ell$. The first option is impossible since it implies that $w\equiv -1\, \pmod \ell$, and thus $k+1\equiv 0\, \pmod \ell$, which is a contradiction. Thus $w^2+w+1\equiv 0\,\pmod \ell$ as desired. The case $w\equiv -k-1 \pmod \ell$ is completely analogous. Conversely, assume that $k$ is a primitive cubic root of unity. We want to see that if $j\in M_k$, then also $kj\in M_k$. This follows from
$$
\langle kj\rangle+ \langle k^2 j\rangle=\langle kj\rangle+ \langle -(k+1) j\rangle=\langle kj\rangle+ \ell-\langle k j\rangle-\langle j\rangle <\ell\,.
$$ 
\end{proof}

\begin{remark}
We can now show that $B_k$ may be taken as the Jacobian of a curve defined over~$\Q$. Assume that $k$ is a primitive cubic root of unity modulo $\ell$ (otherwise $n_k=1$, and there is nothing to prove). Let $t\in \Z$ be such that $1+k+k^2=\ell t$. The automorphism $\nu:=\mu\circ\lambda$ of $\mathcal C$ ($\lambda$ and $\mu$ as in the proof of Lemma \ref{lemma: splitting}) induces an automorphism of $\mathcal C_k$ of order $3$ with two fixed points, which is explicitly given by the assignment
$$
(u,v)\mapsto \left(-\frac{u+1}{u}, (-1)^{k-1}\frac{(1+u)^{t-k}}{u}v^k \right)\,.
$$
Hurwitz formula shows that the genus of the quotient curve of $\mathcal C_k$ by the subgroup generated by $\nu$ has genus $\frac{\ell-1}{6}$. Although $B_k$ can be taken over $\Q$, the isogeny $\Jac(\Cc_k)\sim_F B_k^3$ does not come from an isogeny defined over $\Q$. Indeed, Theorem \ref{theorem: intro} shows that the local factor of $\Cc_k$ at a prime of residue degree $\ell-1$ in $F$ can not be a cube.
\end{remark}

\subsection{Hecke characters attached to the quotient curves $\Cc_k$}\label{section: Hecke characters}

For a prime~$\p$ in~$F$ not lying over~$\ell$ and $x\in \F_\p^*$, with $\F_\p:=\mathcal O_F/\p$, let $\chi_\p(x)\in F^*$ be the only $\ell$th root of unity satisfying the condition
$\pi(\chi_{\p}(x))= x^{\frac{N\p-1}{\ell}}$, where $\pi\colon \mathcal O_F\rightarrow \F_\p$ is the residue map. Then, $\chi_\p\colon \F_\p^*\rightarrow  F^*$ is a character of $\F_\p^*$ of order $\ell$ that we extend to $\F_\p$ by defining $\chi_\p(0):=0$. Let $I_F(\ell)$ stand for the set of all fractional ideals of $\mathcal O_F$ coprime to $\ell$.
Then for any $a \in G$ and $1\leq k \leq \ell-2$, consider the map $J_{{(ka,a)}}\colon I_F(\ell)\rightarrow \C^*$, defined on prime ideals $\p$ of $I_F(\ell)$ by
$$
 J_{{(ka,a)}}(\p):=-\sum_{v\in\F_\p}\chi_\p(v)^{ka}\chi_\p(v+1)^a\,,
$$
and extended to any ideal of $I_F(\ell)$ by the rule
$$
J_{{(ka,a)}}(\mathfrak a\mathfrak b)=J_{{(ka,a)}}(\mathfrak a)J_{{(ka,a)}}(\mathfrak b)\,.
$$
If $\sigma_t\in \Gal(F/\Q)$ and $\p$ is a prime ideal in $I_F(\ell)$, then the following properties are satisfied
\begin{equation}\label{equation: comm}
{}^{\sigma_t}( J_{{(ka,a)}}(\p))= J_{{(ka,a)}}({}^{\sigma_t}\p),\quad
{}^{\sigma_t}( J_{{(ka,a)}}(\p))= J_{{(tka,tk)}}(\p),\quad
J_{{(ka,a)}}(\p)\mathcal O_F=\prod_{j\in M_k}{}^{\sigma_{aj^{-1}}}\p\,.
\end{equation}

It is easy to check the first equality, and the second and third are equalities $(11)$ and $(9)$ of \cite{Wei52}, respectively (see also \cite{Sti90}).
Weil \cite{Wei52} showed that  $J_{{(ka,a)}}$ is a Gr\"ossencharakter of infinity type $aM_k^{-1}:=\{\langle aj^{-1}\rangle\,|\, j \in M_k\}$ and modulus $\ell^2\mathcal O_F$, that is,
\begin{equation}\label{equation: inf type}
J_{{(ka,a)}}(\alpha\mathcal O_F)=\prod_{j\in M_k} {}^{\sigma_{aj^{-1}}}(\alpha)\qquad \text{for all }\alpha\in F^*\text{ with }\alpha\equiv^{\times} 1 \pmod{\ell^2}.\footnote{For $\alpha,\beta \in F^*$ and an ideal $\mathfrak m\subseteq \mathcal O_F$, we write $\alpha \equiv^{\times} \beta \pmod {\mathfrak m}$ if $\alpha$ and $\beta$ are multiplicatively congruent modulo $\mathfrak m$, that is, if $v(\alpha/\beta-1)\geq v(\mathfrak m)$ for every discrete valuation $v$ of $F$.}
\end{equation}
Later Hasse \cite{Has55} showed that the conductor of $J_{{(ka,a)}}$ is $(1-\zeta_\ell)\mathcal O_F$ or $(1-\zeta_\ell)^2\mathcal O_F$ depending on whether $\sum_{j\in M_k}j^{-1}$ is zero or not modulo $\ell$.
The weight of $J_{(ka,a)}$ is $1$, since exactly one of $j$ and $-j$ belongs to $aM_k^{-1}$ and thus $J_{(ka,a)}(\mathfrak a)J_{(ka,a)}(\overline{\mathfrak a})=N\mathfrak a$ for every $\mathfrak a\in I_F(\ell)$.
Moreover, the $\ell'$-adic rational Tate module of $\Jac(\Cc_k)$, simply denoted $V_{\ell'}(\Cc_k)$, admits a decomposition of $G_F$-modules
\begin{equation}\label{equation: del dec}
V_{\ell'}(\Cc_k)=\bigoplus_{a\in G}V_{{(ka,a)}}\,,
\end{equation}
where $V_{{(ka,a)}}$ is a $1$-dimensional $\Q_{\ell'}$-vector space, on which the action of an arithmetic Frobenius $\Frob_\p$ at $\p\nmid \ell$ is by multiplication of $J_{{(ka,a)}}(\p)$ (see Deligne \cite[\S7]{Del82} for a general result on the decomposition of the middle \'etale cohomology group of a Fermat hypersurface). In terms of $L$-functions this amounts to saying that
$$
L((\Cc_k)_{F},s)=\prod_{a\in G}L(J_{{(ka,a)}},s)\,,
\qquad \text{where } 
L(J_{{(ka,a)}},s)=\prod_{\p\nmid\ell}\left(1-\frac{J_{{(ka,a)}}(\p)}{N\p^s}\right)^{-1}\,.
$$

\begin{remark}\label{remark: CM inf type}
This may also be seen as a consequence of the theory of complex multiplication, which predicts that the reflex CM type $M_k^{-1}$ of $\Jac(\Cc_k)$ coincides with the infinity type of the Hecke character attached to its $L$-function.
\end{remark}

\begin{lemma}\label{lemma: Lfunctions}
For any $a\in G$, one has
$$
L((\Cc_k)_{F},s)=L(J_{{(ka,a)}},s)^{\ell-1}\,,\qquad  L(\Cc_k,s)=L(J_{{(ka,a)}},s)\,.
$$
\end{lemma}

\begin{proof} By (\ref{equation: comm}), we have that for every rational prime $p$
$$
\prod_{\p|p}\prod_{t\in G}\left(1-\frac{J_{{(kat,at)}}(\p)}{N\p^{s}}\right)^{-1}=\prod_{\p|p}\prod_{\sigma\in \Gal(F/\Q)}\left(1-\frac{J_{(ka, a)}({}^{\sigma}\p)}{(N{}^{\sigma}\p)^{s}}\right)^{-1}=\prod_{\p|p}\left(1-\frac{J_{(ka, a)}(\p)}{N\p^{s}}\right)^{-(\ell-1)}\,,
$$
from which the first assertion follows. For the second, note that from the Artin formalism we get the equality
$$
L((\Cc_k)_{F},s)=\prod_{\chi\colon\Gal(F/\Q)\rightarrow \C^*}L(\Jac(\Cc_k)\otimes\chi,s)\,.
$$
Therefore, it is enough to show that $L(\Jac(\Cc_k)\otimes\chi,s)=L(\Jac(\Cc_k),s)$ for every character $\chi\colon\Gal(F/\Q)\rightarrow \C^*$, or equivalently, that $V_{\ell'}(\Cc_k)\otimes\chi\simeq V_{\ell'}(\Cc_k)$ as $G_\Q$-modules. For this last isomorphism it is enough to show that 
$$
\Tr V_{\ell'}(\Cc_k)(\Frob_p)=0
$$
for every prime $p\nmid \ell$ and $p\not\equiv 1\pmod \ell$, or equivalently that $|\Cc_k(\F_p)|=p+1$ for every such prime. But it is clear that 
$$
\Cc_k\colon v^\ell=u(u+1)^{\ell-k-1}
$$ 
has exactly $p$ affine points defined over $\F_p$ if $p\not\equiv 1 \pmod\ell$, since in this case exponentiation by $\ell$ is an isomorphism of $\F_p^*$, and then every value of $u\in\F_p$ uniquely determines  the value of $v\in\F_p$.
\end{proof}

\section{The generalized Sato-Tate Conjecture for $\Jac(\Cc_k)$}\label{section: ST group}

In \S\ref{section: simple factors}, we have seen that $\Jac(\Cc_k)$ decomposes over $F$ as $B_k^{n_k}$, where $B_k$ is an abelian variety with complex multiplication by $F_k:=F^{W_k}$ and primitive CM type equal to $M_k/W_k$. The results of this section, will hold for pairs $(\ell,k)$ such that $M_k/W_k$ is non-degenerate, notion that we will now recall. 

\begin{definition}\label{definition: Demjanenko matrix}
For $k\in\{1,\dots,\ell-2\}$ and for $a\in G$, define
$$
E_k(a):=
\begin{cases} 0 & \text{if $a\in M_k$,}\\
1 & \text{if $a\not\in M_k$.}\\
\end{cases}
$$ 
The generalized $k$-Demjanenko matrix is defined as $D_k:=\left(E_k(-c^{-1}a)-\frac{1}{2}\right)_{c,a\in M_k/W_k}$. We will denote the size of $D_k$ by $r_k:=\frac{\ell-1}{2n_k}$ (recall that $n_k=|W_k|$).
\end{definition}

The notion of rank of $M_k/W_k$ was first introduced by Kubota \cite{Kub65}. It is by definition the rank of $\Phi^*_k(\Z[G/W_k])$, where
$$
\Phi^*_k\colon \Z[G/W_k]\rightarrow \Z[G/W_k]\,,\qquad\Phi^*_k([a])=\sum_{b\in M_k/W_k}[b^{-1}][a]=\sum_{c\in G/W_k}E_k(-c^{-1}a)[c]\,. 
$$
Note that $\Phi^*_k$ is well defined precisely because $W_k$ is the subgroup of $G$ fixing $M_k$.
\begin{remark}
Write $X(T_{F_k}):=\Z[G/W_k]$ for the character group of the torus $T_{F_k}$ defined by~$F_k^*$. We may see $\Phi^*_k\colon X(T_{F_k})\rightarrow X(T_{F_k})$ as a map between character groups induced by a map $\Phi_k\colon T_{F_k}\rightarrow T_{F_k}$ between algebraic tori (see \cite{Rib80}). 
\end{remark}

\begin{lemma}[see Lemma 1 of \cite{Kub65}]\label{lemma: image phik} The rank of $M_k/W_k$ is equal to $\rk (D_k)+1$.
\end{lemma}

\begin{proof}
Let $A=(E_k(-c^{-1}a))_{c,a\in G/W_k}$ denote the matrix of $\Phi_k^*$ in the basis $M_k/W_k\cup M_{-k}/W_k$. If we write $U$ for the  $r_k\times  r_k$ matrix whose entries are all ones, we obtain
$$
A=
\begin{pmatrix}
D_k+\frac{1}{2}U & \frac{1}{2}U-D_k\\
\frac{1}{2}U-D_k & D_k+\frac{1}{2}U
\end{pmatrix}\sim
\begin{pmatrix}
D_k+\frac{1}{2}U & \frac{1}{2}U-D_k\\
U  & U
\end{pmatrix}\sim
\begin{pmatrix}
D_k+\frac{1}{2}U & U \\
U  & 2U
\end{pmatrix}\sim
\begin{pmatrix}
D_k & U \\
0  & 2U
\end{pmatrix}\,.
$$
\end{proof}

We say that $M_k/W_k$ is \emph{non-degenerate} if its rank is $r_k+1$, equivalently, if $\rk (D_k)$ is maximal, that is, if $D_k$ has determinant distinct from zero; we say that the pair $(\ell,k)$ is non-degenerate if $M_k/W_k$ is non-degenerate; and we say that $\ell$ is non-degenerate if $(\ell,k)$ is non-degenerate for every $1\leq k\leq \ell-2$. 

\begin{remark}\label{remark: Stark Lenstra} For example, the degenerate primes $\ell$ with $3<\ell < 400$ are

$$
67,\, 127, \, 139,\, 151,\,157,\,163,\,199,\,211,\,223,\,271,\,277,\,283,\,307,\,331,\,367,\,379,\,397\,.
$$
Observe that Theorem \ref{theorem: carac} implies that any prime $\ell\equiv 2\,\pmod{3}$ is non-degenerate. Lenstra and Stark noticed that every sufficiently large prime $\ell \equiv 7 \pmod{12}$ is degenerate (see \cite[p. 354]{Gre80}). In fact, every prime $\ell \equiv 7 \pmod{12}$ distinct from $7$ and $19$ is degenerate (see \cite{FS15} for this and also for an answer to the question concerning the proportion of $1\leq k\leq \ell-2$ for which $(\ell,k)$ is degenerate).
\end{remark}

We refer the reader to \S\ref{section: degenerate primes} for a more explicit characterization of the degenerate pairs $(k,\ell)$.

\subsection{The algebraic Sato-Tate group of $\Jac(\Cc_k)$}

Let us start by fixing some notation. Set
$$
I_2:=
\begin{pmatrix}
1 & 0\\
0 & 1
\end{pmatrix}\,,\qquad
J_2:=
\begin{pmatrix}
0 & 1\\
-1 & 0
\end{pmatrix}\,.
$$
For any $m\geq 1$, the symplectic group $\Sp_{2m}/\Q$ is taken with respect to the symplectic form given by the block matrix
$$
J:=\diag(J_2,\stackrel{m}\ldots,J_2)\,.
$$
We have a diagonal embedding
$$
\iota_{n_k}\colon \Sp_{2r_k}\hookrightarrow \Sp_{\ell-1}\,,\qquad \iota_{n_k}(A)= \diag(A, \stackrel{n_k}\ldots ,A)\,.
$$ 
Recall that $\Jac(\Cc_k)\sim_FB_{k}^{n_k}$. Let us write $A_k:=\Jac(\Cc_k)$ to shorten notation. For a prime $\ell'$, let 
$$
\varrho\colon G_\Q\rightarrow \Aut(V_{\ell'}(A_{k}))\,
$$ 
denote the $\ell'$-adic representation attached to $A_{k}$.
Write $G_{\ell'}:=\varrho(G_\Q)$ and let $G_{\ell'}^{Zar}$ denote its Zariski closure.
Let $G_{\ell'}^1:=G_{\ell'}\cap \Sp_{\ell-1}$ and let $G_{\ell'}^{1,Zar}$ be its Zariski closure. One may also obtain $G_{\ell'}^{1,Zar}$ as $G_{\ell'}^{Zar}\cap \Sp_{\ell-1}$ (see \cite[\S2]{BK11}).

The algebraic Sato-Tate Conjecture for $A_{k}$ predicts the existence of an algebraic subgroup $\AST(A_{k})$ of $\Sp_{\ell-1}/\Q$, called the algebraic Sato-Tate group of $A_{k}$, such that $$
G_{\ell'}^{1,Zar}=\AST(A_{k})\otimes_{\Q}\Q_{\ell'}
$$ for every prime $\ell'$ (see \cite[\S 2.2]{FKRS12}). We now recall the definition of the \emph{twisted Lefschetz group}, which we will denote by $\TL(A_{k})$. For $\tau\in G_\Q$, set
\[
\Lef(A_{k})(\tau) := \{\gamma \in \Sp_{\ell-1}\,|\, \gamma \alpha \gamma^ {-1} = \tau(\alpha) \mbox{ for all $\alpha \in \End((A_{k})_{\Qbar})\otimes\Q$}\}\,,
\]
where $\alpha$ is seen as an endomorphism of $H_1((A_{k})_\C,\Q)$. Then one defines 
$$
\TL(A_{k}) := \bigcup_{\tau \in G_\Q} \Lef(A_{k})(\tau)
$$

\begin{lemma}\label{Lemma: ASTC}  Suppose that $(\ell,k)$ is a non-degenerate pair. The algebraic Sato-Tate Conjecture holds for $A_{k}$ with $\AST(A_{k})=\TL(A_{k})$. 
\end{lemma}

\begin{proof}
We will apply \cite[Thm. 2.16.(a)]{FKRS12}. It is thus enough to show that $\Hg(A_{k})=\Lef(A_{k})$, and that the Mumford-Tate Conjecture\footnote{The Mumford-Tate Conjecture for abelian varieties with CM is known in general, as it follows from the work of Shimura and Taniyama (see Ribet's review on \cite{Ser68}). The recent expository article \cite{Yu13} gives a detailed proof.} holds for $A_{k}$. 
Here, $\Hg(A_k)$ denotes the Hodge group of~$A_k$ and $\Lef(A_k)$ is the \emph{Lefschetz group}, which is $(\Lef(A_{k})(\id))^0$ by definition.   
Deligne \cite[I,Prop. 6.2]{Del82} showed that $G_{\ell'}^{1,Zar,0}\subseteq\Hg(A_{k})\otimes_\Q\Q_{\ell'}$ for every prime $\ell'$. Moreover one has the inclusion $\Hg(A_{k})\subseteq \Lef(A_{k})$ (see \cite[Rem. 4.3]{BK11}). We will prove the two required facts simultaneously by showing that the inclusions
\begin{equation}\label{equation: MT inclusions}
G_{\ell'}^{1,Zar,0}\subseteq\Hg(A_{k})\otimes_\Q\Q_{\ell'}\subseteq \Lef(A_{k})\otimes_\Q\Q_{\ell'}
\end{equation}
are in fact equalities. It is well known that if the first inclusion is an equality for one prime, then it is so for every prime (see \cite{Tan96}). The same is obviously true for the second inclusion.
Since the objects of the chain of inclusions (\ref{equation: MT inclusions}) do no depend on base change by finite extensions, we may replace $A_{k}$ by its base change $(A_{k})_F$ to $F$. But then we may write (\ref{equation: MT inclusions}) as
$$
\iota_{n_k}(G_{\ell'}^{1,Zar,0}(B_{k}))\subseteq\iota_{n_k}(\Hg(B_{k})\otimes_\Q\Q_{\ell'})\subseteq \iota_{n_k}(\Lef(B_{k})\otimes_\Q\Q_{\ell'})\,,
$$
and it thus suffices to show that the inclusions
$$
G_{\ell'}^{1,Zar,0}(B_k)\subseteq\Hg(B_{k})\otimes_\Q\Q_{\ell'}\subseteq \Lef(B_{k})\otimes_\Q\Q_{\ell'}
$$
are in fact equalities for some prime $\ell'$. 
Since $B_{k}$ is simple and has non-degenerate CM type, the results of \cite{BGK03} yield
$$
G_{\ell'}^{Zar,0}(B_k)=\{\diag(x_1,y_1,\dots,x_{r_k},y_{r_k})\,|\, x_i,y_i\in\Q_{\ell'}^*,\,x_1y_1=\dots = x_{r_k}y_{r_k} \}
$$
for every prime $\ell'$ of good reduction for $A_{k}$ that splits completely in $F$.
This implies
\begin{equation}\label{equation: exp}
G_{\ell'}^{1,Zar,0}(B_k)=\{\diag(x_1,y_1,\dots,x_{r_k},y_{r_k})\,|\, x_i,y_i\in\Q_{\ell'}^*,\,x_1y_1=\dots = x_{r_k}y_{r_k}=1 \}\,.
\end{equation}
To compute  $\Lef(B_{k})\otimes_\Q\Q_{\ell'}$ observe that any matrix commuting with any $\beta\in\End(H_1((B_{k})_\C,\Q))$ (as in Lemma \ref{lemma: alpha} below, for example) must be diagonal. Imposing that it preserves $J$, we deduce that
$$
\Lef(B_{k})\otimes_\Q\Q_{\ell'}=\left\{\diag\left(x_1,\frac{1}{x_1},\dots,x_{r_k},\frac{1}{x_{r_k}}\right)\,|\, x_i\in\Q_{\ell'}^* \right\}\,,
$$
 yielding the desired equality. 
\end{proof}

\begin{lemma}\label{lemma: alpha}
Let $\alpha:\Cc_k\rightarrow \Cc_k$ be the automorphism  defined by $\alpha(u,v)=(u,\zeta_\ell v)$. It induces an endomorphism of $H_1((A_{k})_\C,\Q)$ that we also denote by $\alpha$. Then $\beta:=\sum_{w\in W_k}{}^w\alpha$ is an endomorphism of $H_1((B_{k})_\C,\Q)$ and there exist symplectic basis of $H_1((A_{k})_\C,\C)$ and $H_1((B_{k})_\C,\C)$ (with respect to $J$) for which
$$
\alpha=\diag (\{ \Theta_i\}_{i\in \{1,\dots,\frac{\ell-1}{2}\}} )\,,
\qquad
\beta=\diag \big(\big\{ \sum_{w\in W_k}{}^w\Theta_i\big\}_{i\in \{1,\dots,\frac{\ell-1}{2n_k}\}} \big)\,,
$$
where, if $g$ is a generator of $G$, then
$$
\Theta_i:=
\begin{cases}
\{\zeta_\ell^{g^i}, \overline \zeta_\ell^{g^i}\} &\text{if $g^i\in M_k$} \,,\\
\{  \overline\zeta_\ell^{g^i}, \zeta_\ell^{g^i} \} &\text{if $g^i\not \in M_k$\,.}
\end{cases}
$$
\end{lemma}

\begin{proof} According to Remark \ref{remark: basis differentials}, fix the basis $B:=\left\{\omega_j:= \frac{x^{\langle j\rangle-1}}{y^{\langle (k+1)j\rangle}}dx\,|\, j\in M_k\right\}$ of $\pi_k^*\Big(\Omega^1_{\Cc_k/\Q}\Big)$.
Since $\alpha^*\left(\omega_j \right)=\zeta_\ell^{j}\omega_j$, we have that the matrix of $\alpha^*$ in the basis $B$ is 
$\diag(\{\zeta_\ell^{j}\,|\, j\in M_k\})$. The Lemma follows from taking the symplectic basis of $H_1((A_{k})_\C,\C)$ corresponding to $B$ and the symplectic basis of  $H_1((B_{k})_\C,\C)$ corresponding to $\left\{ \omega_j \,|\, j\in M_k/W_k\right\}$.
\end{proof}

\begin{corollary}\label{corollary: group comp} The group of components of $G^{1,Zar}_{\ell'}$ and $\AST(A_{k})$ are isomorphic to $\Gal(F/\Q)$. The connected component $\AST(A_{k})^0$ is isomorphic to  $\AST((A_{k})_F)$.
\end{corollary}

\begin{proof} This follows from Lemma \ref{Lemma: ASTC} (see  \cite[Prop. 2.17]{FKRS12}).
\end{proof}

\subsection{The Sato-Tate group of $\Jac(\Cc_k)$}

\begin{proposition}\label{proposition: ST F} If $(\ell,k)$ is a non-degenerate pair, then
\small
$$
\ST(\Jac(\Cc_k)_F)=\{\iota_{n_k}(\diag(u_1,\overline u_1,\dots,u_{r_k},\overline u_{r_k}))\,|\,u_1,\dots, u_{r_k}\in\Unitary(1)\}\subseteq \USp(\ell-1)\,.
$$
\normalsize
\end{proposition}

\begin{proof}
Let $\ell'$ be a prime and take an embedding $\Q_{\ell'}\hookrightarrow \C$. By definition, $\ST(\Jac(\Cc_k)_F)$ is a maximal compact subgroup of  
$$
\AST((A_{k})_F)\otimes_\Q\C\simeq \AST(A_{k})^0\otimes_\Q\C\simeq G_{\ell'}^{1,Zar,0}\otimes_{\Q_{\ell'}}\C\,.
$$
But it follows immediately from (\ref{equation: exp}) that one can take
$$
\{\iota_{n_k}(\diag(u_1,\overline u_1,\dots,u_{r_k},\overline u_{r_k}))\,|\,u_1,\dots, u_{r_k}\in\Unitary(1)\}
$$
as such a maximal compact subgroup. 
\end{proof}

\begin{proposition}\label{proposition: ST group}
Suppose that $(\ell,k)$ is a non-degenerate pair. Then 
$$
\ST(\Jac(\Cc_k))\simeq \Unitary(1)\times\stackrel{r_k}\ldots\times\Unitary(1)\rtimes (\Z/\ell\Z)^*\,.
$$ 
More precisely, if $g$ is a generator of $G$, let 
$$
\gamma=
\begin{pmatrix}
0 & \Gamma_2 & 0 & \dots & 0 & 0\\
0 & 0 & \Gamma_3 & \dots & 0 & 0\\
\vdots & \vdots & \vdots & \ddots & \vdots & \vdots \\
0 & 0 & 0 & \dots & \Gamma_{\frac{\ell-3}{2}} & 0\\
0 & 0 & 0 & \dots & 0 &  \Gamma_{\frac{\ell-1}{2}}\\
\Gamma_1 & 0 & 0 & \dots & 0 & 0
\end{pmatrix}\,,
\qquad
\text{where}\quad
\Gamma_i=
\begin{cases}
I_2\text{ if $g^{i-1},g^{i}\in M_k$, or  $g^{i-1},g^{i}\not\in M_k$}\\
J_2\text{ if $g^{i-1},-g^{i}\in M_k$ or $-g^{i-1},g^{i}\in M_k$.}
\end{cases}
$$
Then, as a subgroup of $\USp(\ell-1)$, the group $\ST(\Jac(\Cc_k))$ is conjugate to the group generated by
$\ST(\Jac((\Cc_k)_F))$ and  $\gamma$.
\end{proposition}

\begin{proof}
Let $\sigma_g$ be a generator of $\Gal(F/\Q)$.
By Proposition \ref{proposition: ST F} and Corollary \ref{corollary: group comp}, it suffices to prove that:
\begin{enumerate}[i)] 
\item $\gamma\in \ST(\Jac(\Cc_k))$;
\item $\gamma^{\ell-1}\in  \ST(\Jac((\Cc_k)_F))$ but $\gamma^{d}\not \in  \ST(\Jac((\Cc_k)_F))$ for any proper divisor $d$ of $\ell-1$.
\end{enumerate}
For $i)$ it is enough to check that $\gamma\in \Lef(A_k)(\sigma_g)$, but this is true since one easily checks that
$$
\gamma \diag(\{ \Theta_i\}_{i\in \{1,\dots,\frac{\ell-1}{2}\}})\gamma^{-1}=\diag({}^{\sigma_g}\{ \Theta_i\}_{i\in \{1,\dots,\frac{\ell-1}{2}\}})\,
$$
(for this observe that $J_2\begin{pmatrix}u & 0\\ 0& \overline u\end{pmatrix} J_2^{-1}=\begin{pmatrix}\overline u & 0\\ 0& u\end{pmatrix}$). We now show $ii)$. It is obvious that if  $\gamma^{d}\in  \ST(\Jac((\Cc_k)_F))$, then $\frac{\ell-1}{2}$ divides $d$. It is easily checked that
$$
\gamma^{\frac{\ell-1}{2}}=\diag(\prod_{i=1}^{\frac{\ell-1}{2}}\Gamma_i,\stackrel{\frac{\ell-1}{2}}\ldots,\prod_{i=1}^{\frac{\ell-1}{2}}\Gamma_i)=\diag(J_2^{\mathcal N_g},\stackrel{\frac{\ell-1}{2}}\ldots,J_2^{\mathcal N_g})\,,
$$
where $\mathcal N_g:=|\{j\in M_k\,|\, gj\not\in M_k\}|$. The condition $\gamma^{\frac{\ell-1}{2}}\not \in  \ST(\Jac((\Cc_k)_F))$ follows from $\mathcal N_g$ being an odd integer (note that $g$ is not a quadratic residue modulo $\ell$ and then apply Lemma \ref{lemma: Gauss Lemma} below). This also implies $\gamma^{\ell-1}=\diag((-I_2)^{\mathcal N_g},\stackrel{\frac{\ell-1}{2}}\ldots,(-I_2)^{\mathcal N_g})=-I_{\ell-1}\in\ST(\Jac((\Cc_k)_F))$.
\end{proof}

\begin{lemma}[Gauss Lemma]\label{lemma: Gauss Lemma}
Let $a\in G$ and write  $\mathcal N_a:=|\{j\in M_k\,|\, aj\not\in M_k\}|$. Then
$$
(-1)^{\mathcal N_a}=\left( \frac{a}{\ell}\right)\,.
$$
\end{lemma}

\begin{proof}
Consider the product
\begin{equation}\label{equation: gauss1}
Z_k:=\prod_{j\in M_k}a j =a^{\frac{\ell-1}{2}}\prod_{j\in M_k} j \in G\,.
\end{equation}
For $j\in G$, define
$$
|j|_k:=\begin{cases}
\phantom{-} j & \text{ if $j\in M_k$,}\\
-j & \text{ if $j\not \in M_k$.}
\end{cases}
$$ 
Observe that $\{|aj|_k\}_{j\in M_k}=M_k$. Then by definition of $\mathcal N_a$, we have
\begin{equation}\label{equation: gauss2}
Z_k=(-1)^{\mathcal N_a}\prod_{j\in M_k}|aj|_k=(-1)^{\mathcal N_a}\prod_{j\in M_k}  j.
\end{equation}
By comparing (\ref{equation: gauss1}) and (\ref{equation: gauss2}), we obtain
$(-1)^{\mathcal N_a}=a^{\frac{\ell-1}{2}}=\left( \frac{a}{\ell}\right)$.
\end{proof}
As in \cite[Prop. 2.17]{FKRS12}), Lemma \ref{Lemma: ASTC} and Proposition \ref{proposition: ST group} yield the next result.
\begin{corollary}
Let $E/\Q$ be a subextension of $F/\Q$ and let $\sigma_{g^i}$ be a generator of $\Gal(F/E)$ for some $0\leq i\leq \ell-2$. Then $\ST(\Jac((\Cc_k)_E))$ is generated by $\ST(\Jac((\Cc_k)_F))$ and $\gamma^i$.
\end{corollary}

\subsection{Equidistribution statements}\label{section: proof}

Our next goal is to prove the generalized Sato-Tate Conjecture for $\Jac(\Cc_k)$; roughly speaking, the equidistribution of the Frobenius conjugacy classes on $\ST(\Jac(\Cc_k))$ with respect to its Haar measure. 

Let us first recall the paradigm of Serre to prove results of equidistribution. Let $\mathcal G$ be a compact group and $\mathcal X$ be the set of its conjugacy classes. Let $P$ denote an infinite subset of the set of primes of a number field, and let $\{\mathfrak p_i\}_{i\geq 1}$ be an ordering by norm of $P$ (there are in principle many such orderings, but equidistribution statements do not depend of fixing a particular one). Assume given an assigment $\mathcal A\colon \mathfrak p\in P\rightarrow x_{\mathfrak p}\in \mathcal X$. If $\varrho: \mathcal G\rightarrow \GL_n(\C)$ is a representation of $\mathcal G$, write
$$
L_{\mathcal A}(\varrho,s)=\prod_{\p\in P}\det(1-\varrho(x_\p)N\p^{-s})^{-1}\,.
$$

\begin{theorem}[\cite{Ser68}, p. I-23]\label{theorem: Serre} Assume that for every irreducible nontrivial representation~$\varrho$ of~$\mathcal G$ the Euler product $L_{\mathcal A}(\varrho,s)$ converges for $\Re(s)>1$ and extends to a holomorphic and nonvanishing function for $\Re(s)\geq 1$. Then, the sequence $\{x_{\p_i}\}_{i\geq 1}$ is equidistributed over $\mathcal X$ with respect to the projection on $\mathcal X$ of the Haar measure of $\mathcal G$.
\end{theorem}

Returning to our case, let $E/\Q$ be a subextension of $F/\Q$. Denote by $\mathcal G_E$ the group $\ST(\Jac(\Cc_k)_E)$ and let $\mathcal X_E$ the set of conjugacy classes of $\mathcal G_E$. For $\wp$ a prime of $E$, we define a conjugacy class~$x_\wp$ of $\mathcal G_E$ using the isomorphism $\mathcal G_E\simeq \mathcal G_F\rtimes \Gal(F/E)$. Indeed, set
$$
x_{\wp}:=\left(\iota_{n_k}\left(\diag\left(\frac{J_{(ka_1,a_1)}(\p)}{\sqrt {N\p}},\frac{{J_{(ka_1,a_1)}(\overline\p)}}{\sqrt {N\p}},\dots, \frac{J_{(ka_{r_k},a_{r_k})}(\p)}{\sqrt {N\p}},\frac{{J_{(ka_{r_k},a_{r_k})}(\overline\p)}}{\sqrt {N\p}}\right)\right), \Frob_{\wp}\right)\in \mathcal X_E\,,
$$
where $a_1,\dots a_{r_k}$ is a complete system of representatives of $M_k/ W_k$, and $\p$ is a prime of $F$ lying over~$\wp$.\footnote{One may give an explicit matricial description of $x_\wp$ by making use of the results of Br\"unjes \cite[Prop. 11.4]{Bru04}.} 
Let now $\{\wp_i\}_{i\geq 1}$ be an ordering by norm of the set of primes of good reduction for $(\Cc_k)_E$. Define the assigment 
$$
\mathcal A_E\colon \{\wp_i\}_{i\geq 1}\rightarrow \mathcal X_E \,,\qquad \wp_i\mapsto x_{\wp_i}\,.
$$

\begin{conjecture}[Generalized Sato-Tate]\label{conjecture: ST} 
The sequence $x_E:=\{x_{\wp_i}\}_{i\geq 1}$ is equidistributed on $\mathcal X_E$ with respect to the image on $\mathcal X_E$ of the Haar measure of $\mathcal G_E$.
\end{conjecture}

\subsection{Equidistribution over $\Q(\zeta_\ell)$}

\begin{theorem}\label{theorem: ST F} Assume that $(\ell,k)$ is non-degenerate. 
The sequence $x_F:=\{x_{\p_i}\}_{i\geq 1}$ is equidistributed on $\mathcal X_F$, i.e. Conjecture \ref{conjecture: ST} holds for $\Jac(\Cc_k)_F$.
\end{theorem}

\begin{proof}
Note that the group $\mathcal G_F$ is isomorphic to
$$
\Unitary(1)\times \stackrel{r_k}\ldots \times\Unitary(1).
$$
The irreducible representations of $\Unitary(1)\times \stackrel{r_k}\ldots\times \Unitary(1)$ are the characters
\begin{equation}\label{equation: irred rep}
\phi_{b_1,\dots,b_{r_k}}\colon \Unitary(1)\times \stackrel{r_k}\ldots\times\Unitary(1)\rightarrow \C^*,\qquad \phi_{b_1,\dots,b_{r_k}}(z_1,\dots,z_{r_k})=\prod_{i=1}^{r_k}z_i^{b_i}\,,
\end{equation}
where $b_1,\dots,b_{r_k}\in \Z$. By Theorem \ref{theorem: Serre}, we have to prove that for any $b_1,\dots,b_{r_k}\in \Z$, not all of them zero, the $L$-function
$$
L_{\mathcal A_F}(\phi_{b_1,\dots, b_{r_k}},s)=\prod_{i\geq 1}\left(1-\frac{J_{(ka_1,a_1)}(\p_i)^{b_1}\cdot\dots\cdot J_{(ka_{r_k},a_{r_k})}(\p_i)^{b_{r_k}}}{\sqrt {N\p_i}^{b_1+\dots+b_{r_k}}}N\p_i^{-s  }\right)^{-1}
$$
is holomorphic and nonvanishing for $\Re(s)\geq 1$.
But, up to a finite number of local Euler factors, this is just the $L$-function $L(\Psi,s)$ of the unitarized Gr\"ossencharakter 
$$
\Psi:=\frac{J_{(ka_1,a_1)}(\cdot)^{b_1}\cdot\dots\cdot J_{(ka_{r_k},a_{r_k})}(\cdot)^{b_{r_k}}}{\sqrt {N(\cdot)}^{b_1+\dots+b_{r_k}}}\,.  
$$
By unitarized we mean that it takes values in $\Unitary(1)\subseteq \C^*$ and not just in $\C^*$. Hecke \cite{Hec20} showed that the $L$-function of a nontrivial unitarized Gr\"ossencharakter is holomorphic and nonvanishing for $\Re(s)\geq 1$.  Therefore, it only remains to prove that the Gr\"ossencharakter $\Psi$ is nontrivial. Suppose it were, and let $B:=\frac{b_1+\dots+b_{r_k}}{2}$. For every prime $\p$ of $F$, reindexing the $b_j$'s on the set $M_k/W_k$ for notation simplicity, we have that (\ref{equation: comm}) implies
\begin{equation}\label{equation: inf dem}
\mathcal O_F=\Psi(\p)\mathcal O_F=\prod_{t\in G}\frac{\prod_{j\in M_k/W_k}({}^{\sigma_t}\p)^{b_j E_k(-t^{-1}j)}}{({}^{\sigma_t}\p)^B}\,.
\end{equation}
It follows from Artin's Theorem on independence of characters that for every $t\in G$, we have $\sum_{j\in M_k/W_k}b_jE_k(-t^{-1}j)=B$. 
But this implies that $\det(D_k)=0$, which is a contradiction with the assumption of $(\ell,k)$ being non-degenerate.
\end{proof}

The proof above extends naturally to the case of an absolutely simple abelian variety with complex multiplication (see \cite[\S3.5]{Fit14}).

\subsection{Irreducible representations of $\Unitary(1)\times\stackrel{r_k}\ldots\times \Unitary(1)\rtimes G$}\label{section: irred reps}

In \S\ref{section: over Q} we will prove the generalized Sato-Tate Conjecture for $\Jac(\Cc_k)$ over $\Q$ when $(\ell,k)$ is a non-degenerate pair.  
Thus we are led by Theorem \ref{theorem: Serre} to the study of the irreducible representations of  $\mathcal G_\Q\simeq\Unitary(1)\times \stackrel{r_k}\ldots\times \Unitary(1)\rtimes G$, where the action of a generator $g$ of $G$ on $\Unitary(1)\times \stackrel{r_k}\ldots\times \Unitary(1)$ is given by the rule
$$
\iota_{n_k}({}^{g}\diag(u_1\overline u_1,\dots, u_{r_k},\overline u_{r_k}))=\gamma \iota_{n_k}(\diag(u_{1},\overline u_1,\dots, u_{r_k},\overline u_{r_k}))\gamma^{-1}\,.
$$
To shorten notation, we will write
$$
\mathcal G_0:=\Unitary(1)\times \stackrel{r_k}\ldots\times \Unitary(1)\,,\qquad \mathcal G:=\mathcal G_0\rtimes G\,.
$$

We now follow \cite[\S 8.2]{Ser77}, where the irreducible representations of a semidirect product by an abelian group are characterized.
For any character $\phi_{b_1,\dots,b_{r_k}}$ of $\mathcal G_0$ as in (\ref{equation: irred rep}), let $H_{b_1,\dots,b_{r_k}}\subseteq G$ be the subgroup such that
\begin{equation}\label{equation: def H}
\phi_{b_1,\dots,b_{r_k}}(u_1,\dots, u_{r_k})= \phi_{b_{1},\dots,b_{r_k}}({}^h(u_1,\dots, u_{r_k}))\qquad \text{ for every }h\in H_{b_1,\dots,b_{r_k}}\,.
\end{equation}
Write  $\mathcal H:=\mathcal G_0\rtimes H_{b_1,\dots,b_{r_k}}$. One has that
$$
\phi_{b_1,\dots,b_{r_k}}\colon\mathcal H\rightarrow \C^*,\qquad \phi_{b_1,\dots,b_{r_k}}(z_1,\dots,z_{r_k},h)=\prod_{i=1}^{r_k}z_i^{b_i}
$$
is a character of $\mathcal H$. Then by \cite[Prop. 25]{Ser77} every irreducible representation of $\mathcal G$ is of the form
$\theta:=\Ind_{\mathcal G}^{\mathcal H}(\chi \otimes \phi_{b_1,\dots,b_{r_k}})$, where $\chi$ is a character of $H_{b_1,\dots,b_{r_k}}$ that we may view as a character of $\mathcal H$ by composing with the projection $\mathcal H\rightarrow H_{b_1,\dots,b_{r_k}}$.

\subsection{Equidistribution over $\Q$}\label{section: over Q}

In this section, we write $\doteq$ to indicate equality of $L$-functions up to a finite number of local Euler factors.

\begin{theorem}\label{theorem: ST Q} Assume that $(\ell,k)$ is non-degenerate. The sequence $x_\Q:=\{x_{p_i}\}_{i\geq 1}$ is equidistributed on $\mathcal X_\Q$, i.e. Conjecture \ref{conjecture: ST} holds for $\Jac(\Cc_k)$.
\end{theorem}

\begin{proof}
Let $\theta=\Ind_{\mathcal G}^{\mathcal H}(\chi \otimes \phi_{b_1,\dots,b_{r_k}})$ be as in \S\ref{section: irred reps}. We have to show that 
$$
L_{\mathcal A_\Q}(\theta,s)=\prod_{i\geq 1, p_i\subseteq \Z}\det\left(  1-\theta(x_{p_i}) p_i^{-s}\right)^{-1}
$$
is holomorphic and nonvanishing for $\Re(s)\geq 1$. Let us first consider the case $\chi=1$. Write $n:=|H_{b_1,\dots,b_{r_k}}|$.
By the Artin formalism, we have that
$$
L_{\mathcal A_F}(\phi_{b_1,\dots,b_{r_k}},s)=L_{\mathcal A_\Q}(\Ind^{\mathcal H}_{\mathcal G}\Ind^{\mathcal G_0}_{\mathcal H}\phi_{b_1,\dots,b_{r_k}},s)
=L_{\mathcal A_\Q}(n\Ind^{\mathcal H}_{\mathcal G}\phi_{b_1,\dots,b_{r_k}},s)
=L_{\mathcal A_\Q}(\theta,s)^{n}\,.
$$
The second equality follows from (\ref{equation: def H}).
But in the proof of Theorem \ref{theorem: ST F}, we have seen that $L_{\mathcal A_F}(\phi_{b_1,\dots,b_{r_k}},s)\doteq L(\Psi,s)$ is holomorphic and nonvanishing for $\Re(s)\geq 1$, from which the desired result follows. For the general case ($\chi$ not necessarily trivial), let $\tilde\chi$ be a character of $G$ such that $\tilde\chi|_{H_{b_1,\dots,b_{r_k}}}=\chi$ (the existence of $\tilde \chi$ is guaranteed by the fact of $G$ being cyclic). Then $\theta=\Ind_{\mathcal G}^{\mathcal H}(\chi \otimes \phi_{b_1,\dots,b_{r_k}})=\tilde\chi\otimes \Ind_{\mathcal G}^{\mathcal H}( \phi_{b_1,\dots,b_{r_k}})$. The cyclicity of $G$ additionally implies that
$$
n\theta=\tilde \chi\otimes \Ind^{\mathcal G_0}_{\mathcal G}\phi_{b_1,\dots,b_{r_k}}=\Ind^{\mathcal G_0}_{\mathcal G}\phi_{b_1,\dots,b_{r_k}}\,,
$$
from which $ L(\Psi,s)\doteq L_{\mathcal A_\Q}(\theta,s)^{n}$ follows again.
\end{proof}

We have thus shown that Conjecture \ref{conjecture: ST} holds for $E=F$ and $E=\Q$. The choice $E=\Q$ in \S\ref{section: irred reps} and \S\ref{section: over Q} was made only for the purpose of simplifying the exposition; the proof of Theorem~\ref{theorem: ST Q} can be immediately generalized for an arbitrary intermediate extension $E/\Q$ of $F/\Q$.

\section{Vanishing of the determinant of a Demjanenko matrix}\label{section: degenerate primes}

For non-degenerate pairs $(\ell,k)$, one may explicitly determine the limiting distributions of the coefficients of the normalized local factors attached to $\Cc_k$ from the results of \S\ref{section: ST group}.

In \S\ref{section: explicit distributions}, we will describe an alternative method to compute these limiting distributions. The interest of this method relies on the fact that, for a degenerate pair $(\ell,k)$, there still exist some residue degrees exhibiting a ``non-degenerate behavior''. As a consequence, we will be able to compute the limiting distributions when we restrict to primes of such residue degrees. These residue degrees $f$ may be characterized by the fact that the rank of $D_{k,f}$ is maximal, where $D_{k,f}$ is a generalization of the Demjanenko matrix $D_k$.

The goal of this section is to provide the technical results for the method that will be presented in \S\ref{section: explicit distributions} (essentially Theorem \ref{theorem: det deg}, see below). To a certain extent, it is an independent section and for the reader exclusively interested in determinants of Demjanenko matrices, it should suffice to review (\ref{equation: MW}) and  Definition \ref{definition: Demjanenko matrix}, and skip everything else from the previous sections. Conversely, the reader exclusively interested in equidistribution questions concerning quotients of the Fermat curves may just look at Definitions \ref{definition: Ekf} and \ref{definition: Dkf}, assume Theorem \ref{theorem: det deg} and Proposition \ref{proposition: equi pep}, and ignore the rest of the section at a first reading. 

\begin{definition}\label{definition: Ekf}
Let $E_k$ be as in Definition \ref{definition: Demjanenko matrix}. For a divisor $f$ of $\ell-1$, let $H_f$ be the subgroup of $G$ of order $f$, and for $a\in G$ let 
$$
E_{k,f}(a):=\sum_{h\in H_f}E_k(ah)\,. 
$$
Define also 
$
W_{k,f}:=\{w\in G\,|\,E_{k,f}(a)=E_{k,f}(aw), \forall a\in G\}
$ 
and let $n_{k,f}$ denote its cardinality.
\end{definition}

We will use the following notation: for every subset $X\subseteq G$, we will denote by $[X]$ the element $\sum_{x\in X}[x]\in\Z[G]$. Observe that 
\begin{equation}\label{equation: carac}
W_{k,f}=\left\{w\in G\,|\,[H_f][M_k^{-1}]([w]-[1])=0\text{ in }\Z[G] \right\}\,,
\end{equation}
since we have the following equalities
$$
[H_f][M_k^{-1}]=\sum_{j\in G} E_{k,f}(-j^{-1})[j]\,,\quad [H_f][M_k^{-1}][w]=\sum_{j\in G} E_{k,f}(-j^{-1}w)[j]\,.
$$
\begin{definition}\label{definition: Dkf} 
The generalized $(k,f)$-Demjanenko matrix is $D_{k,f}=\left(E_{k,f}(-c^{-1}a)-\frac{f}{2}\right)_{c,a\in M_k/W_{k,f}}$. We will denote the size of $D_{k,f}$ by $r_{k,f}:=\frac{\ell-1}{2n_{k,f}}$. 
\end{definition}

\begin{remark}
Observe that $E_{k,1}=E_k$, $W_{k,1}=W_k$, $D_{k,1}=D_k$, $n_{k,1}=n_k$, $r_{k,1}=r_k$, $H_f\subseteq W_{k,f}$, and that $E_{k,f}(-a)=f-E_{k,f}(a)$. 
\end{remark}

\begin{lemma}\label{lemma: cons wfk}
A divisor $f$ of $\ell-1$ is even if and only if $W_{k,f}=G$. In this case, we have $D_{k,f}=0$.
\end{lemma}

\begin{proof}
Since $f$ is even, then $-1\in H_f$, but this means that $[H_f][M_k^{-1}]=\frac{f}{2}[G]$, from which one implication of the lemma is clear in virtue of (\ref{equation: carac}). For the other implication first note that 
$$
|H_f\cap M_k|=\sum_{h\in H_f}E_k(h)=E_{k,f}(1),\qquad |H_f\cap -M_k|=\sum_{h\in H_f}E_k(-h)=E_{k,f}(-1)\,.
$$
If $W_{k,f}=G$, then $E_{k,f}(1)=E_{k,f}(-1)$, and thus $|H_f\cap M_k|=|H_f\cap -M_k|$. It follows that $f=|H_f\cap M_k|+|H_f\cap -M_k|$ is even.
\end{proof}

Before proceeding to prove Theorem \ref{theorem: det deg}, we need four auxilliary results: Propostions \ref{proposition: demjanenko determinant}, \ref{proposition: alaRib}, \ref{proposition: calculet brillant}, and~\ref{proposition: equi pep}. We first introduce some notation. For an abelian group $A$, let $X(A)$ denote the group of characters of $A$. For every odd divisor $f$ of $\ell-1$, let $X^-_{k,f}(G)$ (resp. $X^+_{k,f}(G)$) denote the set of odd (resp. even) characters of $G$ that are trivial on $W_{k,f}$. Write simply $X^-_k(G)$ and $X^+_k(G)$ when $f=1$. 

\begin{proposition}\label{proposition: demjanenko determinant}
 For $1\leq k\leq \ell-2$ and $f$ an odd divisor of $\ell-1$, we have 
$$\det(D_{k,f})=\left(\frac{-f}{2}\right)^{r_{k,f}}\prod_{\chi\in X^-_{k,f}(G)}\sum_{a\in M_k/W_{k,f}}\chi(a)=\left(\frac{-f}{2n_{k,f}}\right)^{r_{k,f}}\prod_{\chi\in X^-_{k,f}(G)}\sum_{a\in M_k}\chi(a)\,.$$ 
\end{proposition}

\begin{proof}
We will apply the Dedekind determinant formula (DDF)\footnote{Recall that for a finite abelian group $A$ and $\mathcal F$ a function on $A$, the DDF establishes that
$$
\det(\mathcal F(ab^{-1}))_{a,b\in A}=\prod_{\psi\in X(A)}\sum_{a\in A}\psi(a)\mathcal F(a)\,.
$$}, following the strategy considered in \cite{Haz90}, \cite{Doh94}, or \cite{SS95}. Let us write
$$
\delta_k(a):=
\begin{cases}
1 & \text{if }a\in M_k\,,\\
-1 & \text{if }a\not\in M_k\,.
\end{cases}
$$
Note that by Lemma \ref{lemma: cons wfk}, $X^-_{k,f}(G)$ is non-empty and thus there is a bijection between $X^-_{k,f}(G)$ and  $X^+_{k,f}(G)$. Choose $\omega\in X^-_{k,f}(G)$. Observe that the function $\mathcal F(a)=\sum_{h\in H_f}\omega(ha)\delta_k(ha)$, for $a\in G$,
is well-defined on $G/(\{\pm 1\}W_{k,f})$, since $\omega(a)\delta_k(a)=\omega(-a)\delta_k(-a)$ and 
$$
\mathcal F(a)=\omega(a)(f-2E_{k,f}(a))=\omega(wa)(f-2E_{k,f}(wa))=\mathcal F(wa)
$$
for every $w\in W_{k,f}$. Then we have
\begin{align*}
\prod_{\chi\in X^-_{k,f}(G)}\sum_{a\in M_k} \chi(a) &=\prod_{\chi\in X^+_{k,f}(G)}\sum_{a\in M_k} \chi(a)\cdot\omega(a) \\[4pt]
  &=\prod_{\chi\in X^+_{k,f}(G)}\sum_{a\in M_k} \chi(a)\cdot \omega(a)\delta_k(a)\\[4pt]
  &=\prod_{\psi\in X(G/(\{\pm 1\}W_{k,f}))}\sum_{a\in M_k/W_{k,f}} \frac{n_{k,f}}{f}\psi(a)\cdot \mathcal F( a)\\[4pt]
  &=\det\left(\frac{n_{k,f}}{f}\mathcal F( a b^{-1})\right)_{a,b\in M_k/W_{k,f}}\\[4pt]
   &=\left(\frac{-n_{k,f}}{f}\right)^{r_{k,f}}\det(\omega( a b^{-1})(2E_{k,f}( -a b^{-1})-f))_{a,b\in M_k/W_{k,f}}\,.\\[4pt]
\end{align*}
Multiplying the $a$-row of the matrix by $\omega(a)^{-1}$ and the $b$-column by $\omega(b)$ for every $a,b\in M_k/W_{k,f}$ cancels out the factor $\omega(a b^{-1})$ without changing the determinant.
\end{proof}

Kubota \cite[Lem. 2]{Kub65} showed that the rank of $D_k$ is the number of characters $\chi\in X_k^-(G)$ for which the sum $\sum_{a\in M_k/W_k}\chi(a)$ is nonzero.
We will now show that an analogous statement holds true when we consider $D_{k,f}$.
To this end, we will extend Ribet's proof \cite[Prop. 3.10]{Rib80} of the result of Kubota. Define the map
$$
\Phi^*_{k,f}\colon \Z[G/W_{k,f}]\rightarrow \Z[G/W_{k,f}]\,,\qquad\Phi^*_{k,f}([a])=[H_f][M_k^{-1}][a]=\sum_{c\in G/W_{k,f}}E_{k,f}(-c^{-1}a)[c]\,. 
$$
Note that $\Phi^*_{k,f}$ is well defined precisely because of the definition of $W_{k,f}$.

\begin{lemma}\label{lemma: image phikf} The rank of $\Phi_{k,f}^*(\Z[G/W_{k,f}])$ is $\rk( D_{k,f})+1$.
\end{lemma}

\begin{proof}
The proof goes exactly as in Lemma \ref{lemma: image phik}.
\end{proof}

\begin{proposition}\label{proposition: alaRib}
The rank of $D_{k,f}$ is the number of characters $\chi\in X_{k,f}^-(G)$ for which the sum $\sum_{a\in M_k/W_{k,f}}\chi(a)$ is nonzero.
\end{proposition}

\begin{proof}
Consider the basis vectors $v_\chi:=\sum_{c\in G/W_{k,f}}\chi(c)[c]$ of $\C[G/W_{k,f}]$, where $\chi$ runs over the set $X(G/W_{k,f})\simeq X_{k,f}(G)$. Observe that
$$
\Phi_{k,f}^*(v_\chi)=f\left( \sum_{a\in M_k}\chi(a)\right)v_{\chi}\,.
$$ 
And one concludes by noting that the only even character for which $\sum_{a\in M_k/W_{k,f}}\chi(a)\not =0$ is the trivial one.
\end{proof}

\begin{proposition}\label{proposition: calculet brillant}
For $\psi\in X^-_k(G)$, one has
$$
\sum_{a\in M_k/W_k}\psi(a)=\frac{1}{n_{k}}\sum_{a\in M_k}\psi(a)= \frac{B_{1,\psi}}{n_k}\left( \frac{1}{\psi(k+1)}-1-\frac{1}{\psi(k)} \right)\,.
$$
Here, $B_{1,\psi}:=\frac{1}{\ell}\sum_{a=1}^{\ell-1}\psi(a) a$ stands for the first generalized Bernoulli number\footnote{It is well known that $B_{1,\psi}\not=0$, due to the Analytic Class number Formula (see (\ref{equation: ACNF}) in the proof of Proposition \ref{proposition: formula detkf}).}.
\end{proposition}

\begin{proof}
From a theorem of Stickelberger, Greenberg \cite{Gre80} derived the equality
\begin{equation}\label{equation: stick}
\sum_{a\in M_k}[a^{-1}]=\frac{1}{\ell}([1]+[k]-[1+k])\sum_{a\in G} \langle a\rangle [-a^{-1}]\,,
\end{equation}
of elements in $\Z[G]$. By evaluating it at an odd character $\psi^{-1}$ of $G$, he obtained
\begin{equation}\label{equation: Gre}
\sum_{a\in M_k}\psi(a)=B_{1,\psi}\left( \frac{1}{\psi(k+1)}-1-\frac{1}{\psi(k)} \right)\,,
\end{equation}
from which the statement of the proposition follows immediately. Nevertheless, we would like to present an alternative proof of (\ref{equation: Gre}) by generalizing the method used by Leopold \cite{Le62} to deal with the case $k=1$. We will write the nonzero number $\ell B_{1,\psi}$ in the two different following ways. First, using that $\langle (k+1)a\rangle=\langle a\rangle +\langle ak\rangle$ for every $a\in M_k$, we obtain 
\begin{align}\label{equation: first way}
\ell B_{1,\psi}&= \sum_{a\in M_k}\psi((k+1)a)\langle (k+1)a\rangle + \sum_{a\in M_k}\psi(-(k+1)a)\langle -(k+1)a\rangle \\[4pt]
&=2\sum_{a\in M_k}\psi((k+1)a)\langle (k+1)a\rangle-\psi(k+1)\sum_{a\in M_k}\psi(a)\ell\\[4pt]
&=2\psi(k+1)\left( \sum_{a\in M_k} \psi(a)\langle a\rangle + \sum_{a\in M_k}\psi (a)\langle ka\rangle\right)-\psi(k+1)\sum_{a\in M_k}\psi(a)\ell\,.
\end{align}
Secondly, we have
\begin{align}\label{equation: second way}
\ell B_{1,\psi}&= \sum_{a\in M_k}\psi(a)\langle a\rangle - \sum_{a\in M_k}\psi(a)(\ell -\langle a\rangle )=2\sum_{a\in M_k}\psi(a)\langle a\rangle -\sum_{a\in M_k}\psi(a)\ell\,.
\end{align}
Subtracting $\psi(k+1)$ times equation (\ref{equation: second way}) from equation (\ref{equation: first way}), we obtain 
\begin{align*}
(1-\psi(k+1))\ell B_{1,\psi}&=2\psi(k+1)\sum_{a\in M_k}\psi(a)\langle a k\rangle \\[4pt]
&=\frac{2\psi(k+1)}{\psi(k)}\sum_{a\in M_k}\psi(ak)\langle a k\rangle\\[4pt]
&=\frac{\psi(k+1)}{\psi(k)}\left(\ell B_{1,\psi}+\psi(k)\sum_{a\in M_k}\psi(a)\ell \right)\\[4pt]
&=\frac{\psi(k+1)}{\psi(k)}\ell B_{1,\psi}+\psi(k+1)\ell \sum_{a\in M_k}\psi(a)\,.
\end{align*}
This yields
$$
\sum_{a\in M_k}\psi(a)= B_{1,\psi}\left( \frac{1}{\psi(k+1)}-1-\frac{1}{\psi(k)}\right)\,. 
$$
\end{proof}

\begin{proposition}\label{proposition: equi pep}
For $\ell>3$, $1\leq k\leq \ell-2$, and every odd divisor $f$ of $\ell-1$, the following statements are equivalent:
\begin{enumerate}[i)]
\item $\sum_{a\in M_k}\psi_0(a)^f=0$, where $\psi_0$ denotes a generator of $X(G)$;
\item $
k^{2f}+k^f+1\equiv 0\,(\bmod \ell)$ and $(k+1)^fk^f\equiv -1(\bmod \ell)\,;
$
\item $H_f\subsetneq W_{k,f}$. 
\end{enumerate}
In this case, $k\in W_{k,f}$, $W_{k,f}$ is generated by $H_f$ and $k$, and $|W_{k,f}|=3f$.
\end{proposition}

\begin{proof}
Assume that $i)$ holds. By (\ref{equation: Gre}), we have that 
\begin{equation}\label{equation: root unity rel}
\frac{1}{\psi_0^f(k+1)}=1+\frac{1}{\psi_0^f(k)} \,.
\end{equation}
But it is a trivial fact that if  $\omega$ is a root of unity, then $\omega+1$ is a root of unity if and only if $\omega$ is a primitive cubic root of unity. Thus, taking $\omega=1/\psi_0^f(k)$, equation (\ref{equation: root unity rel}) is equivalent to the assertion that $\psi_0^f(k)$ is a primitive cubic root of unity and $\psi_0^f(k+1)\psi_0^f\left(-k\right)=1$. The injectivity of $\psi_0$ implies $ii)$. 

Assume now $ii)$. Note that $k,k^2\not \in H_f$ (otherwise $k^{2f}+k^f+1\equiv 3 \pmod \ell$), while $k^3\in H_f$ (since $k^{3f}-1\equiv (k^{f}-1)(k^{2f}+k^f+1)\pmod \ell$). We will show that $k\in W_{k,f}$, equivalently by~(\ref{equation: stick}) that  
\begin{equation}\label{equation: igual1}
\ell[H_f][M_k^{-1}]=[H_f]([1]+[k]-[k+1])\left( \sum_{a\in G}\langle a\rangle [-a^{-1}]\right)
\end{equation}
is invariant by multiplication by $[k]$.
Observe that for every $b\in G$, we have
$$
[H_f]([b]+[-b])\left( \sum_{a\in G}\langle a\rangle [-a^{-1}]\right)=[H_f]\left(\ell\sum_{a\in G}[a]\right)= f\ell\sum_{a\in G}[a]\,.
$$
Thus, taking $b=k+1$, the right hand side of (\ref{equation: igual1}) is equal to
\begin{equation}\label{equation: igual2}
[H_f]([1]+[k]+[-k-1])\left( \sum_{a\in G}\langle a\rangle [-a^{-1}]\right)-f\ell\sum_{a\in G}[a]\,.
\end{equation}
Observe that for every $a,b\in G$, we have $
[a][H_f]=[b][H_f]\Leftrightarrow ab^{-1}\in H_f\,.
$
We claim that $k^2(-k-1)^{-1}\in H_f$. Indeed, first note that $(k+1)^{3f}\equiv -1 \pmod \ell$, and then observe that
$$
(k^2(-k-1)^{-1})^f\equiv k^{2f}(k+1)^{2f}\equiv k^{3f}\equiv 1\pmod \ell\,.
$$
Thus (\ref{equation: igual2}) is equal to
$$
[H_f]([1]+[k]+[k^2])\left( \sum_{a\in G}\langle a \rangle[-a^{-1}]\right)-f\ell\sum_{a\in G}[a]
$$
for which it is clear that each of its two terms are invariant by multiplication by $[k]$.  

Finally, assume $iii)$. For $w\in W_{k,f}$, one has by definition $[M_k^{-1}][H_f][w]=[M_k^{-1}][H_f]$. Evaluating this equality at $\psi_0^{-f}$, we obtain that
$$
\sum_{a\in M_k}\psi_0(a)^f=\left( \sum_{a\in M_k}\psi_0(a)^f\right)\psi_0(w^{-1})^f\,.
$$
Hence, if $w\not \in H_f$, then $\psi_0(w^{-1})^{f}\not = 1$, and $\sum_{a\in M_k}\psi_0(a)^f=0$.

We still want to see that $|W_{k,f}|=3f$ for an $f$ as in the statement. Then $3f$ does not satisfy $ii)$, and thus by $iii)$ we have $W_{k,3f}=H_{3f}$. From the inclusions
$$
H_f\subsetneq W_{k,f}\subseteq W_{k,3f}=H_{3f}\,,
$$
we obtain that $W_{k,f}=H_{3f}$. 
\end{proof}

Let $\mathcal F_0$ denote the set of (odd) divisors $f$ of $\ell-1$ such that $W_{k,f}=H_{3f}$ (or equivalently the set of odd divisors $f$ such that any of the conditions $i)$, $ii)$, or $iii)$ of Proposition \ref{proposition: equi pep} hold).

\begin{theorem}\label{theorem: det deg}
For $\ell>3$ and $1\leq k\leq \ell-2$, the following conditions are equivalent:
\begin{enumerate}[i)]
\item $\det(D_k)=0$;
\item $k$ is not a primitive cubic root of unity modulo $\ell$ and $\mathcal F_0$ is non-empty;
\item If $v_3$ denotes the $3$-adic valuation and $\ord$ denotes the order in $G$:
\begin{enumerate}[a)]
\item $k$ is not a primitive cubic root of unity modulo $\ell$;
\item $\ord(-k^2-k)$ and $\ord(k)$ are odd;
\item $v_3( \ord(k)) > v_3(\ord(-k^2-k))$.
\end{enumerate}
\end{enumerate}
In this case, $\mathcal F_0$ is exactly the set of odd divisors of $\ell-1$ that are multiples of $N_k/3$ but not of $N_k$, where $N_k:=\Lcm (\ord(-k^2-k),\ord(k))$, and thus $\rk(D_k)= \frac{\ell-1}{2}\left(1-\frac{2}{N_k}\right)$. 
\end{theorem}

\begin{proof}
If $\det(D_k)=0$, then by Proposition \ref{proposition: demjanenko determinant}, there exists an odd divisor $f$ of $\ell-1$ such that $i)$ of Proposition \ref{proposition: equi pep} holds. Moreover, $k$ can not be a primitive cubic root of unity, since then $ii)$ of Proposition \ref{proposition: equi pep} would not hold.

If $k$ is not a primitive cubic root of unity, then $n_k=1$ by Lemma \ref{lemma: card W}. Then, if $\mathcal F_0$ is non-empty, $i)$ of Proposition \ref{proposition: equi pep} implies that $\det(D_k)=0$ by Proposition \ref{proposition: demjanenko determinant}.  

Note that the two equations of $ii)$ of Proposition \ref{proposition: equi pep} can be replaced by the following three:
\begin{equation}\label{equation: replace}
k^f\not \equiv 1\pmod \ell,\qquad k^{3f}\equiv 1\pmod \ell,\qquad (-k^2-k)^f\equiv 1\pmod \ell\,.
\end{equation} 
Clearly, there exists an odd divisor $f$ of $\ell-1$ satisfying (\ref{equation: replace}) if and only if $a)$, $b)$, and $c)$ hold. Moreover, in this case, an odd divisor $f$ verifies (\ref{equation: replace}) if and only if $f$ is an odd multiple of  
$$
f_0:=\Lcm \left(\ord(-k^2-k),\frac{\ord(k)}{3}\right)
$$
that is not a multiple of $3f_0$. Since the set of such $f$'s is in bijection with the set
$$
\left\{1\leq m \leq \frac{\ell-1}{f_0}\,|\,(m,6)=1 \right\}\,,
$$
the number of such $f$'s is $\frac{\ell-1}{3f_0}=\frac{\ell-1}{N_k}$. But by Proposition \ref{proposition: alaRib}, this number is precisely $\dim(\Ker(D_k))$.

\end{proof}

The following result tells us that the rank of $D_k$ is  ``asymptotically non-degenerate".

\begin{proposition}
For $\ell$ prime, we have
$$
\lim_{\begin{array}{c}\ell \rightarrow \infty\\ 1\leq k\leq\ell-2 \end{array}}\frac{\rk(D_k)}{r_k}=1\,.
$$
\end{proposition}

\begin{proof} When $\ell$ is non-degenerate, the quotient $\rk(D_k)/r_k=1$ and there is nothing to prove. Since when $(\ell,k)$ is degenerate, one has $r_k=\frac{\ell-1}{2}$, by Theorem \ref{theorem: det deg} it is enough to show that $N_k\rightarrow \infty$ when $\ell\rightarrow \infty$ (the existence of infinitely many degenerate primes is ensured by Remark \ref{remark: Stark Lenstra}). Given an integer $N_0 >0$, we want to show that there exists a prime $\ell_0> 0$ such that for every $\ell > \ell_0$ and every $1\leq k \leq \ell-2$, one has that $N_k> N_0$. This immediately follows from the claim that for every $N_0$, the set $S_{N_0}$ of degenerate primes $\ell$ such that $N_0=N_k$ for some $1\leq k \leq \ell-2$ has finite cardinality. Indeed, set $f:=N_0/3$ and define the polynomials
$$
p_f(x):=x^{2f}+x^f+1\,,\qquad q_f(x):=\frac{(x+1)^fx^f+1}{x^2+x+1}\in \Z [x]\,.
$$
Clearly, $S_{N_0}$ is a subset of the set of primes dividing the resultant $R_{N_0}$ of $p_f(x)$ and $q_f(x)$, and therefore it suffices to show that $R_{N_0}$ is nonzero. This may be deduced from the fact that the roots of $p_f(x)$ are unrepeated roots of unity, whereas $q_f(x)$ has neither double roots nor roots of finite order. 
\end{proof}

\begin{remark}
We claim that $N_k\geq 27$. This implies that 
$$
\rk(D_k)\geq \frac{25}{27}\cdot \frac{\ell-1}{2}\,,
$$
which is a slightly better bound than the one computed in \cite{Mai89}. This bound is sharp, since $N_k=27$ for $\ell=271$ and $k=32$. Recalling the notation of the previous proof, to show that $N_k\geq 27$, it is enough to observe that $S_9=S_{15}=S_{21}=\emptyset$. This follows from the fact that
$$
R_9=3^4\,,\qquad R_{15}= 5^{10}\,,\qquad R_{21}=7^{16}\,,
$$
and none of $3$, $5$, or $7$ is degenerate. Note that $R_{27}=3^{16}\cdot 271^6$ and thus $S_{27}=\{271\}$. 
\end{remark}

We will say that a divisor $f$ of $\ell-1$ is a $k$-degenerate residue degree if $\det(D_{k,f}) = 0$ and that $f$ is non-$k$-degenerate otherwise (however, to easy the terminology, we will drop the $k$ from now on). By Lemma \ref{lemma: cons wfk}, if $f$ is even, then $f$ is degenerate. For a degenerate pair $(\ell,k)$, let $\mathcal F_1$ denote the set of odd divisors $f$ of $\ell-1$ such that $v_3(f)\geq v_3(N_k)$.

\begin{proposition}\label{proposition: char degenerate res} Let $f$ be an odd divisor of $\ell-1$. Then:
\begin{enumerate}[i)]
\item If $(\ell,k)$ is degenerate, then $f$ is a non-degenerate residue degree if and only if $f\in \mathcal F_0 \cup \mathcal F_1$.
Moreover, in case $f\in \mathcal F_1$ we have $W_{k,f}=H_{f}$, whereas in case $f \in \mathcal F_0$ we have $W_{k,f}=H_{3f}$. If $f$ is degenerate, then $W_{k,f}=H_f$.
\item If $(\ell,k)$ is non-degenerate, then $f$ is non-degenerate. Moreover, $W_{k,f}=H_{3f}$ if $k$ is a primitive cubic root of unity and $v_3(f)=0$, and $W_{k,f}=H_{f}$ otherwise.
\end{enumerate}
\end{proposition}

\begin{proof}
To shorten notation, let us write $\Sigma_f$ for $\sum_{a\in M_k}\psi_0(a)^{f}$. Assume that $(\ell,k)$ is degenerate. If $f\in \mathcal F_1$, then no multiple of $f$ lies in $\mathcal F_0$. Thus $W_{k,f}=H_f$ and $\Sigma_{f'}\not = 0$ for every multiple $f'$ of $f$. Then Proposition \ref{proposition: demjanenko determinant} yields $\det(D_{k,f}) \not= 0$. If $f\in \mathcal F_0$, then $W_{k,f}=H_{3f}$. Since $\Sigma_{f'}\not = 0$ for every multiple $f'$ of $3f$, Proposition \ref{proposition: demjanenko determinant} yields $\det(D_{k,f}) \not= 0$.
If $f\not \in \mathcal F_0 \cup \mathcal F_1$, then $v_3(f)< v_3(N_k)$ and $f$ is not a multiple of $N_k/3$. Since $f\not\in \mathcal F_0$, we have $W_{k,f}=H_{f}$. Note that there exists a multiple $f'$ of $f$ in $\mathcal F_0$. Then, Proposition \ref{proposition: demjanenko determinant} yields $\det(D_{k,f}) = 0$.

Assume that $(\ell,k)$ is non-degenerate. Let $f$ be an odd divisor of $\ell-1$. If $k$ is not a primitive cubic root of unity, then $W_k$ is trivial. Then, by Proposition \ref{proposition: demjanenko determinant}, $\det(D_k)\not=0$ implies that $f$ is non-degenerate. In this case, $W_{k,f}=H_f$. Suppose now that $k$ is a primitive cubic root of unity. If $v_3(f)>0$, then $W_{k,f}=H_f$ and for every multiple $f'$ of $f$ we have $\Sigma_{f'}\not = 0$. Thus~$f$ is non-degenerate. If $v_3(f)=0$, then $W_{k,f}=H_{3f}$ and for every multiple $f'$ of $3f$ we have $\Sigma_{f'}\not = 0$. Thus $f$ is also non-degenerate.
\end{proof}

The characterization of \emph{non-degenerate residue degree} and the description of $n_{k,f}$ given in the Introduction follow from the previous proposition.

\begin{example} If $(\ell,k)=(67,6)$, then $\det(D_k)=0$. Since $N_k=33$, we have that $\det(D_{k,3})$, $\det(D_{k,11})$, and $\det(D_{k,33})$ are nonzero, that is, $3$, $11$, and $33$ are non-degenerate residue degrees. Similarly, if $(\ell,k)=(163,10)$, then $\det(D_k)=0$, but $\det(D_{k,27})$ and $\det(D_{k,81})$ are nonzero, since $N_k=81$.
\end{example}

\begin{remark}
In Proposition \ref{proposition: char degenerate res}, we have characterized the divisors $f$ of $\ell-1$ for which $\det(D_{k,f})$ vanishes. It is easy to see that, in this case, 
$$
\rk(D_{k,f})= \frac{\ell-1}{2f}\left(1-\frac{2f}{N_{k,f}}\right)\,,
$$ 
where $N_{k,f}:=\Lcm (\ord(-k^2-k),\ord(k),f)$. 
\end{remark}

Finally, we would like to give formulas of $\det(D_{k,f})$ in terms of the relative class number (for $k=1$, similar formulas are given in \cite{Haz90}, \cite{SS95}, or \cite{Doh94}).

\begin{proposition}\label{proposition: formula detkf}
Let $f$ be an odd divisor of $\ell-1$. Set $F_{k,f}:=F^{W_{k,f}}$, let $\omega_{k,f}$ be the number of roots of unity contained in $F_{k,f}$, and let $h^-_{k,f}$ be the relative class number of $F_{k,f}$. One has
$$
\det(D_{k,f})=\frac{ h^-_{k,f}P_{k,f}}{\omega_{k,f}}\left(\frac{f}{n_{k,f}}\right)^{r_{k,f}}\,,\quad\text{where}\quad P_{k,f}:=\prod_{\psi\in X^-_{k,f}(G)}\left( \frac{1}{\psi(k+1)}-1-\frac{1}{\psi(k)}\right)\,.
$$
\end{proposition}

\begin{proof}
Since $f$ is odd, $F_{k,f}$ is a CM field and the Analytic Class Number Formula  (ACNF) states that 
\begin{equation}\label{equation: ACNF}
h^{-}_{k,f}=\omega_{k,f}\prod_{\psi\in X^-(\Gal(F_{k,f}/\Q))}\frac{-1}{2}B_{1,\psi}
\end{equation}
(see \cite{La78}). Since $X(\Gal(F_{k,f}/\Q)\simeq X^-(G/W_{k,f})=X^-_{k,f}(G)$, we obtain the statement combining ACNF with (\ref{equation: Gre}) and Proposition \ref{proposition: demjanenko determinant}.
\end{proof}

\begin{remark}
One can show that the rank of the image of $\Phi_k^*$ is the dimension of the Mumford-Tate group of $A_k:=\Jac(\Cc_k)$ (see \cite{Yu13}). One thus has that $\rk(D_k)=\dim \Hg(A_k)$.
For $1\leq k,k'\leq \ell-1$, let $A_{k,k'}:=\Jac(\Cc_k)\times\Jac(\Cc_{k'})$ and one has analogously that $\dim \Hg(A_{k,k'})=\rk(D_k|D_{k'})\leq \frac{\ell-1}{2}$. If we choose $k,k'$ in order that $A_k$ and $A_{k'}$ are nonisogenous and of non-degenerate type (we can do this in virtue of Theorems \ref{theorem: KR} and \ref{theorem: det deg}), this yields an example in which
$$
\Hg(A_{k,k'})\subsetneq \Hg(A_k)\times \Hg(A_{k'})\qquad \Lef(A_{k,k'})=\Lef(A_k)\times \Lef(A_{k'})\,.
$$
In particular, the Algebraic Sato-Tate Conjecture, which holds for $A_k$ (resp. $A_{k'}$) taking the algebraic group $\TL(A_k)$ (resp. $\TL(A_{k'})$), does not hold for $A_{k,k'}$ taking the algebraic group $\TL(A_{k,k'})$.
\end{remark}

\section{Computing distributions explicitly}\label{section: explicit distributions}

As mentioned in \S\ref{section: degenerate primes}, the results of \S\ref{section: ST group} suffice to determine the distributions of the coefficients of the normalized local factors attached to $\Cc_k$ when $(\ell,k)$ is non-degenerate. Nonetheless, in this section we will present an alternative direct method to compute them, which works independently. It is based on an accurate description of the local factors $L_p(\Cc_k,T)$, and we emphasize that it may be applied without the necessity of computing the Sato-Tate group. We will compute distributions restricting to primes of a fixed residue degree in $F$. We encounter a curious phenomenon: even in the cases in which the pair $(\ell,k)$ is degenerate, the method of this section permits one to describe the distributions once one restricts to either even or non-degenerate residue degrees (i.e. the necessarily odd residue degrees $f$ such that $\det(D_{k,f})\not = 0$).

\subsection{Local factors of $\Jac(\Cc_k)$}

We will describe the local factors of $\Jac(\Cc_k)$ in terms of the subgroup $W_{k,f}\subseteq G$, where $f$ is a divisor of $\ell-1$.   

\begin{lemma}\label{lemma: Wkf}
Let $\p$ be a prime of $F$ of residue degree $f|\ell-1$ and coprime to $\ell$. For any $a\in G$, we have
$$
W_{k,f}=\{w\in G\,|\, {}^{\sigma_w}(J_{(ka,a)}(\p))=J_{(ka,a)}(\p)\}\,.
$$
\end{lemma}

\begin{proof}
By \cite[Lemma~3.2]{Gon99}, the right hand side of the above equation coincides with
$$
\{w\in G\,|\, {}^{\sigma_w}(J_{(ka,a)}(\p)\mathcal O_F)=J_{(ka,a)}(\p)\mathcal O_F\}\,.
$$
Now the lemma follows from
$$
{}^{\sigma_w}(J_{(ka,a)}(\p)\mathcal O_F)=J_{(ka,a)}(\p)\mathcal O_F\quad \Longleftrightarrow \quad 
\prod_{j\in G/H_f}{}^{\sigma_j}\p^{E_{k,f}(-aj^{-1})}=\prod_{j\in G/H_f}{}^{\sigma_j}\p^{E_{k,f}(-waj^{-1})}\,.
$$
\end{proof}

\begin{proposition}\label{proposition: local factor}
Let $p\not =\ell$ be a prime of residue degree $f$ in $F$. For any prime $\p$ of $F$ lying over~$p$, one has
$$
L_p(\Cc_k,T)=\prod_{a\in G/W_{k,f}}(1-J_{(ka,a)}(\p)T^f)^{\frac{n_{k,f}}{f}}\,.
$$
\end{proposition}

\begin{proof}
In terms of local factors, Lemma \ref{lemma: Lfunctions} states
$$
L_p(\Cc_k,T)^{\ell-1}=L_\p((\Cc_k)_F,T^f)^{\frac{\ell-1}{f}}\,.
$$
But, by Lemma \ref{lemma: Wkf}, we have
$$
L_\p((\Cc_k)_F,T^f)^{\frac{\ell-1}{f}}=\prod_{a\in G/W_{k,f}}(1-J_{(ka,a)}(\p)T^f)^{n_{k,f}\frac{\ell-1}{f}}\,.
$$
Since $L_p(\Cc_k,T)$ and $\prod_{a\in G/W_{k,f}}(1-J_{(ka,a)}(\p)T^f)^{\frac{n_{k,f}}{f}}$ have both constant term equal to $1$, the proposition follows.
\end{proof}
From the previous proposition, we can deduce an alternative proof of a lemma of \cite{GR78}.
\begin{corollary}[see Lemma 1.1. of \cite{GR78}]\label{corollary: GR}
If $p\not =\ell$ is a prime of even residue degree $f$ in $F$, we have that
$$
L_p(\Cc_k,T)=(1+p^{\frac{f}{2}}T^f)^{\frac{\ell-1}{f}}\,.
$$
\end{corollary}
\begin{proof}By Lemma \ref{lemma: cons wfk}, if $f$ is even, then $W_{k,f}=G$. Therefore, by Proposition \ref{proposition: local factor}, it suffices to show that $J_{(k,1)}(\p)=-p^{\frac{f}{2}}$. Since $-1\in H_f\subseteq W_{k,f}$, we have that $J_{(k,1)}(\p)$ is fixed by complex conjugation and thus real. Then $||J_{(k,1)}(\p)||=p^f$, leaves the two possibilities $J_{(k,1)}(\p)=\varepsilon p^{\frac{f}{2}}$, with $\varepsilon=\pm 1$. To solve the ambiguity, we will compute the number of $\F_{p^f}$-rational points of 
$$
\Cc_k\colon v^\ell=u(u+1)^{\ell-k-1}\,.
$$
One the one hand, this number is
$$
|\Cc_k(\F_{p^f})|=1+p^f-\varepsilon p^{\frac{f}{2}}(\ell-1)\,.
$$ 
On the other hand, the number of $\F_{p^f}$-rational points is $\equiv 3 \pmod \ell$. Indeed for any $u_0\not=0,-1$ there are $\ell$ possibilities for an affine point of the form $(u_0,v)$, there are the two affine points $(0,0)$ and $(-1,0)$, and there is the point at infinity. Since $p^{\frac{f}{2}}\equiv -1 \pmod \ell$, we obtain
$$
3\equiv 2-\varepsilon \pmod\ell\,,
$$
which yields $\varepsilon=-1$.
\end{proof}

Recall that $F_{k,f}$ is the subfield of $F$ fixed by $W_{k,f}$. Let $I_{F_{k,f}}(\ell)$ denote the group of fractional ideals of $F_{k,f}$ that are coprime to $\ell$. Consider the map
$$
\Psi_{(ka,a),f}\colon I_{F_{k,f}}(\ell)\rightarrow F_{k,f}^*\,,\qquad \Psi_{(ka,a),f}(\Pp):=\prod_{w\in W_{k,f}}J_{(wka,wa)}(\Pp\mathcal O_F)\,.
$$

\begin{proposition}\label{proposition: low gros} The homomorphism $\Psi_{(ka,a),f}$ is a Gr\"ossencharacter of infinity type 
$$\sum_{j\in G/W_{k,f}}\frac{n_{k,f}^2}{f}E_{k,f}(-aj^{-1})[j]\in\Z[G/W_{k,f}]
$$ 
and weight $n_{k,f}^2$. If $n_{k,f}>1$, the conductor of $\Psi_{(ka,a),f}$ is $\mathfrak l$, where $\mathfrak l$ is the prime of $F_{k,f}$ lying above $\ell$. For every prime $\Pp\in I_{F_{k,f}}(\ell)$, one has
\begin{equation}\label{equation: facil}
\Psi_{(ka,a),f}(\Pp)\mathcal O_{F_{k,f}}=\prod_{t\in G/W_{k,f}}{}^{\sigma_t}(\Pp)^{\frac{n_{k,f}^2}{f}E_{k,f}(-at^{-1})}\,.
\end{equation}
Moreover, for every prime $\p\in I_{F}(\ell)$ of residue degree $f$ and $\Pp:=\p \cap \mathcal O_{F_{k,f}}$, one has $$\Psi_{(ka,a),f}(\Pp)=J_{(ka,a)}(\p)^{\frac{n_{k,f}^ 2}{f}}\,.$$
\end{proposition}

\begin{proof} For $\alpha\in F_{k,f}^*$ such that $\alpha\equiv^{\times} 1 \pmod{\ell^2}$, by (\ref{equation: inf type}), we have
$$
\begin{array}{lll}
\Psi_{(ka,a),f}(\alpha\mathcal O_{F_{k,f}})&=&\displaystyle{\prod_{w\in W_{k,f}}\prod_{t\in G}{}^{\sigma_t}(\alpha)^{E_k(-wat^{-1})}}\\[4pt]
&=&\displaystyle{ \prod_{t\in G}{}^{\sigma_t}(\alpha)^{\frac{n_{k,f}}{f}E_{k,f}(-at^{-1})}}\\[4pt]
&=&\displaystyle{\prod_{t\in G/W_{k,f}}{}^{\sigma_t}(\alpha)^{\frac{n_{k,f}^2}{f}E_{k,f}(-at^{-1})}}\,.
\end{array}
$$
Equality (\ref{equation: facil}) follows from a similar calculation. To compute the weight, use that we have that $\frac{n_{k,f}^2}{f}E_{k,f}(a)+\frac{n_{k,f}^2}{f}E_{k,f}(-a)=n_{k,f}^2$. If $\p\in I_{F}(\ell)$ has residue degree $f$ and $\Pp:=\p \cap \mathcal O_{F_{k,f}}$, then by definition of $W_{k,f}$, we have
$$
\Psi_{(ka,a),f}(\Pp)=\prod_{w\in W_{k,f}}{}^{\sigma_w}J_{(ka,a)}(\Pp\mathcal O_F)=\prod_{w\in W_{k,f}}J_{(ka,a)}(\p)^{\frac{n_{k,f}}{f}}=J_{(ka,a)}(\p)^{\frac{n_{k,f}^2}{f}}\,.
$$
Since $(1-\zeta_\ell)^2\mathcal O_F$ is a modulus for $J_{(ka,a)}$ (see \S\ref{section: Hecke characters}) and $\mathfrak l\mathcal O_F=(1-\zeta_\ell)^{n_{k,f}}\mathcal O_F$, we have that $\mathfrak l$ is a modulus for $\Psi_{(ka,a),f}$ if $n_{k,f}>1$.
\end{proof}

\begin{theorem}\label{theorem: equidist f} Let $f$ be an (odd) divisor of $\ell-1$ such that $\det(D_{k,f})\not =0$. Let $\{\p_i\}_{i\geq 1}$ be an ordering by norm of the primes of $F$ of residue degree $f$. Set
\begin{equation}\label{equation: seq J}
v_{\p_i} = \left(\frac{J_{(ka_1,a_1)}(\p_i)}{\sqrt{N\p_i}},\dots,\frac{J_{(ka_{r_{k,f}},a_{r_{k,f}})}(\p_i)}{\sqrt{N\p_i}} \right) \,,
\end{equation}
where $a_1,\dots, a_{r_{k,f}}$ is a system of representatives of $M_k/W_{k,f}$. Then the sequence $\left\{ v_{\p_i}\right\}_{i\geq 1}$ is equidistributed over $\Unitary(1)\times\stackrel{r_{k,f}}\ldots \times \Unitary(1)$.
\end{theorem}

\begin{proof} Observe that replacing $v_{\p_i}$ by
$$
v'_{\p_i}=\left(\frac{J_{(ka_1,a_1)}(\p_i)^{\frac{n_{k,f}^2}{f}}}{\sqrt{N\p_i}^{\frac{n_{k, f}^2}{f}}},\dots,\frac{J_{(ka_{r_{k,f}},a_{r_{k,f}})}(\p_i)^ {\frac{n_{k, f}^ 2}{f}}}{\sqrt{N\p_i}^{\frac{n_{k, f}^2}{f}}} \right)
$$
in the statement of the theorem yields an equivalent statement. Set $\Pp_i:=\p_i \cap \mathcal O_{ F_{k,f}}$. By Proposition~\ref{proposition: low gros}, the tuple $v'_{\p_i}$ is exactly
$$
\left(\frac{\Psi_{(ka_1,a_1),f}(\Pp_i)}{\sqrt{N\Pp_i}^ { n_{k,f}^ 2}},\dots,\frac{\Psi_{(ka_{r_{k,f}},a_{r_{k,f}}),f}(\Pp_i)}{\sqrt{N\Pp_i}^{ n_{k,f}^ 2}} \right)\,.
$$
Let us now make a change of notation: $\{ \Pp_i\}_{i\geq 1}$ is an ordering by norm of \emph{all} of the primes of $F_{k,f}$. Let $\Frob_{\Pp_i}\in \Gal(F/F_{k,f})$ be the relative Frobenius at $\Pp_i$. Then the theorem follows from the claim that the sequence
$$
\left\{\left(\frac{\Psi_{(ka_1,a_1),f}(\Pp_i)}{\sqrt{N\Pp_i}^ {n_{k,f}^2}},\dots,\frac{\Psi_{(ka_{r_{k,f}},a_{r_{k,f}}),f}(\Pp_i)}{\sqrt{N\Pp_i}^{n_{k,f}^ 2}},\Frob_{\Pp_i} \right)\right\}_{i\geq 1}
$$
is equidistributed over $\Unitary(1)\times\stackrel{r_{k,f}}\ldots \times \Unitary(1)\times \Gal(F/F_{k,f})$. For $(b_1,\dots,b_{r_{k,f}})\in \Z^{r_{k,f}}$ and $\chi\in \Gal(F/F_{k,r})$, let us write
$$
\Psi:=\Psi_{(ka_1,a_1),f}^{b_1}\cdot\dots\cdot \Psi_{(ka_{r_{k,f}},a_{r_{k,f}}),f}^{b_{r_{k,f}}}\chi\,.
$$
By Theorem \ref{theorem: Serre}, we have to show that
$$
\prod_{i\geq 1}\left(1-\frac{\Psi(\Pp_i)}{\sqrt {N\Pp_i}^{(b_1+\dots+b_{r_{k,f}})n_{k,f}^ 2}}N\Pp_i^{-s  }\right)^{-1}
$$
is holomorphic and nonvanishing for $\Re(s)\geq 1$, unless $\chi$ is trivial and all the $b_i$'s are zero. Artin reciprocity guarantees the existence of an ideal $\mathfrak c$ of $F_{k,f}$ (divisible by precisely the primes that ramify in $F$) such that for all $\alpha\in F_{k,f}^*$ satisfying $\alpha\equiv^{\times} 1 \pmod {\mathfrak c}$, $\chi(\alpha\mathcal O_{F_{k,f}})=1$. This means that $\Psi$ is again a Gr\"ossencharacter\footnote{At this point it becomes apparent the importance of Proposition \ref{proposition: low gros}: it permits to reduce an equidistribution problem about eigenvalues of Gr\"ossencharacters of the field $F$ to an equidistribution problem about eigenvalues of Gr\"ossencharacters of the field $F_{k,f}$, where the primes of residue degree $f$ can be detected by means of another Gr\"ossencharacter.}, and by Hecke's result all we have to check is that it is not trivial unless $\chi$ is trivial and all the $b_i$'s are zero. 
Suppose that $\Psi$ is trivial and write $B:=(b_1+\dots+b_{r_{k,f}}) n_{k,f}^ 2/2$. Then, reindexing the $b_j$'s on the set $ M_k/W_{k,f}$, for every prime $\Pp$ of $\mathcal O_{F_{k,f}}$ with $\Frob_\Pp=1$, equation (\ref{equation: facil}) implies 
\begin{equation}
\begin{array}{lll}
\mathcal O_{F_{k,f}}=\Psi(\Pp)\mathcal O_{F_{k,f}}&=&\displaystyle{\prod_{j\in M_k/W_{k,f}}\prod_{t\in G/W_{k,f}}{}^{\sigma_t}(\Pp)^{b_j\frac{n_{k,f}^2}{f}E_{k,f}(-t^{-1}j)-B}}\,.
\end{array}
\end{equation}
It follows from Artin's Theorem (applied to the set of characters on the monoid generated by prime ideals $\Pp$ of $\mathcal O_{F_{k,f}}$ with $\Frob_\Pp=1$) that for every $t\in G/W_{k,f}$, one has
$$
\sum_{j\in M_k/W_{k,f}}b_j E_{k,f}(-t^{-1}j)-\frac{B f}{n_{k,f}^ 2}=0\,.
$$
Since $\det(D_{k,f})\not=0$, all the $b_j$'s are zero. Then $\Psi=\chi$ and also $\chi$ is trivial.
\end{proof}

We can now deduce Theorem \ref{theorem: intro}. 

\begin{proof}[Proof of Theorem \ref{theorem: intro}] 
Let $\{p_i \}_{i\geq 1}$ be the ordering by size of the rational primes of residue degree $f$ in $F$, and define the assignment 
$$
\{ p_i\}_{i\geq 1}\rightarrow \Unitary(1)\times\stackrel{r_{k,f}}\ldots \times \Unitary(1)\stackrel{\mathrm{proj}}\longrightarrow [-2,2]^{r_{k,f}}/\mathfrak{S}_{r_{k,f}}\,\qquad 
p_i\mapsto v_ {\p_i}\mapsto (s_1(\p_i),\dots, s_{r_{k,f}}(\p_i))\,,
$$
where $v_{\p_i}$ is as in (\ref{equation: seq J}), $\p_i$ is any prime of $F$ lying over $p_i$, and $s_{j}(p_i)=-\frac{J_{(ka_j,a_j)}(\p_i)}{\sqrt{N\p_i}}- \frac{J_{(ka_j,a_j)}(\overline\p_i)}{\sqrt{N\p_i}}$. Note that this assignment is independent of the choice of $\p_i$ over $p_i$.
By Proposition \ref{proposition: local factor}, we have 
$$
L_{p_i}(\Cc_k,T/\sqrt {p_i})=\prod_{j=1}^{r_{k,f}}(1+s_{j}(p_i)T^f+T^{2f})^{\frac{n_{k,f}}{f}}\,.
$$
Observe that the image of the Haar measure of $\Unitary(1)\times\stackrel{r_{k,f}}\ldots \times \Unitary(1)$ on $[-2,2]^{r_{k,f}}/\mathfrak{S}_{r_{k,f}}$ by the map $\mathrm{proj}$ is $
\prod_{i=1}^{r_{k,f}}\frac{1}{\pi}\frac{dx_i}{\sqrt{4-x_i^2}}$.
One concludes by applying Theorem \ref{theorem: equidist f}.
\end{proof}

\subsection{Explicit distributions}\label{section: explicit dist}

For every $p\nmid \ell$, define $a_{i}(p)$, for $i=0,\dots, {\ell-1}$, to be the $i$th coefficient of the \emph{normalized} local factor of $\Cc_k$ at $p$
$$
L_p(\Cc_k,T/\sqrt p):=\sum_{i=0}^{\ell-1}a_{i}(p) T^ i\,.
$$
Note that $a_{i}(p)\in I_i:=\left[ -\binom{\ell-1 }{i}, \binom{\ell-1}{i}\right]$ and that $a_{i}(p)=a_{\ell-1-i}(p)$.
Let $a_i$ denote the sequence $\{a_i(p)\}_{p\nmid \ell}$, where the primes are ordered by size. In this section, from Theorem \ref{theorem: intro}, we will describe how to compute, in the case that $(\ell,k)$ is non-degenerate, the measure~$\mu_i$ over~$I_i$ with respect to which $a_i$ is equidistributed for $i=0,\dots, \frac{\ell-1}{2}$.
Let $a_{i,f}$ be the subsequence of $a_i$ made up of those $a_i(p)$ such that $p$ has residue degree $f$ in $F$. Denote by $\mu_{i,f}$ the measure over $I_i$ with respect to which $a_{i,f}$ is equidistributed. Then 
$$
\mu_i=\sum_{1\leq f|\ell-1} \frac{\varphi( f)}{\ell-1}\mu_{i,f} \quad \text{and}\quad
\M_n[\mu_i]=\sum_{1\leq f|\ell-1} \frac{\varphi( f)}{\ell-1}M_n[\mu_{i,f}]\,,
$$
where $\varphi$ is the Euler function, and $\M_n[\mu_i]$ (resp. $\M_n[\mu_{i,f}]$) stands for the $n$th moment of $\mu_i$ (resp. $\mu_{i,f}$).
We now show how to compute $\mu_{i,f}$ and $\M_n[\mu_{i,f}]$ when $f$ is even or non-degenerate (this covers the case $(\ell,k)$ non-degenerate).

\textbf{Case $f$ even.} For $p\not =\ell$ with residue degree $f$ in $F$, Corollary \ref{corollary: GR} implies that if $i\equiv 0\pmod f$, then
$$
a_i(p)=\binom{\frac{\ell-1}{f}}{\frac{\ell-1-i}{f} }\qquad \text{and thus}\qquad \M_n[\mu_{i,f}]=\binom{\frac{\ell-1}{f}}{\frac{\ell-1-i}{f} }^n\,.
\text{}
$$
If $i\not\equiv 0\pmod f$, then $a_i(p)=0$ and $\M_n[\mu_{i,f}]=0$. Note that this is independent of $(k,\ell)$ being degenerate or not.

\textbf{Case $f$ non-degenerate.} Using Theorem \ref{theorem: intro}, it is a straightforward computation to obtain the first moments $\M_n[\mu_{i,f}]$. In Table \ref{table: moments} in \S\ref{section: numerical data}, we have listed the first moments $\M_n[\mu_i]$ for some non-degenerate pairs $(\ell,k)$.

\section{Examples and numerical data}\label{section: numerical data}

\subsection{Examples}

On Table \ref{table: ST groups} we show $\ST(\Jac(\Cc_k))$ for several non-degenerate pairs $(\ell,k)$. We also write the set $M_k$, the subgroup $W_k$ of $G$, and a generator $g$ of $G$. We use the following notations. We denote by $U$ a random element in the connected component $\ST(\Jac(\Cc_k))^0$ and by $\gamma$ a generator of the group of components. For $u_i\in \Unitary(1)$, write
$$
U_i:=\begin{pmatrix}
u_i & 0\\
0 & \overline u_i
\end{pmatrix}\,.
$$
We denote by $P_{\gamma^i}(T)$ the characteristic polynomial of a random element in $\ST(\Jac(\Cc_k))^0\gamma^i$, and write $s_i:=u_i+\overline u_i$. For every nontrivial divisor $d$ of $\ell-1$, on Table \ref{table: ST groups} we show $P_{\gamma^i}(T)$, for some $\gamma^i$ having order $d$. For the case $d=1$, it is clear that
$$
P_{\gamma^{\ell-1}}(T)=\left( \prod_{i=1}^{r_k}(T^2+s_iT+1)\right)^{n_k}\,.
$$

\subsection{Numerical data}

Recall the notation of \S\ref{section: explicit dist}, for which $a_1$ denotes the sequence of normalized Frobenius traces of $\Jac(\Cc_k)$ and $\mu_1$ the measure on $I_1=\left[-\binom{\ell-1}{2}, \binom{\ell-1}{2}\right]$ attached to it. For $x\gg 0$, let
$$
M_{n,x}:=\frac{1}{\pi(x)}\sum_{p\leq x}\left(\frac{|\Cc_k(\F_p)|-p-1}{\sqrt p}\right)^n\,,
$$
where $\pi(x)$ denotes the number of primes of good reduction $p$ for $\Cc_k$ such that $p\leq x$. Theorem~\ref{theorem: ST F} implies that, for every $n\geq 0$, we have
$$
\M_n[\mu_{1}]=\lim_{x\to\infty} M_{n,x}\,.
$$
On Table \ref{table: moments}, we display the first values of $\M_n[\mu_i]$, for $1\leq i \leq 6$ and even $2\leq n\leq 8$, following the procedure described in \S\ref{section: explicit dist}. On Table \ref{table: statistical moments} we show the first values of $M_{n,x}$ for $x=2^{27}$.

\begin{table}[h]
\begin{center}
\caption{Some Sato-Tate groups $\ST(\Jac(\Cc_k))$. }\label{table: ST groups}
\vspace{5pt}
\begin{tabular}{|l|l|r|r|r|}\hline
$(\ell,k)$ & $M_k$ & $W_k$ & $g$ & $\{g,\dots,g^{\frac{\ell-1}{2}}\,|\, g^{\frac{\ell-1}{2}},\dots,g^{\ell-1} \}$ \\\hline
$(5,2)$ & $\{1,3\}$ & $\{1\}$  &  $ 2 $ & $\{2,4\,|\,\textbf{3},\textbf{1}\}$\\\hline

\multicolumn{5}{|c|}{
$
\gamma=
\begin{pmatrix}
0 & I_2\\
J_2 & 0 
\end{pmatrix}\qquad
U=
\begin{pmatrix}
U_1 &  0\\
0 & U_2 
\end{pmatrix}
$
}\\[6pt]
\multicolumn{5}{|c|}{$P_{\gamma^2}(T)=(T^2+1)^2,\quad P_{\gamma}(T)=T^4+1\,.$}\\\hline
$(7,2)$ & $\{1,2,4\}$ & $\{1,2,4\}$ &  $  3 $ &  $\{3,\textbf 2,6\,|\,\textbf 4, 5,\textbf 1\}$\\\hline

\multicolumn{5}{|c|}{
$
\gamma=
\begin{pmatrix}
0 & J_2 & 0\\
0 & 0 & J_2\\
J_2 & 0 & 0
\end{pmatrix}\qquad
U=
\begin{pmatrix}
U_1 & 0 & 0\\
0 & U_1 & 0\\
0 & 0 & U_1
\end{pmatrix}
$
}\\[6pt]
\multicolumn{5}{|c|}{$P_{\gamma^3}(T)=(T^2+1)^3,\quad P_{\gamma^2}(T)=T^6+(u_1^3+\overline u_1^3)T^3+1,\quad P_{\gamma}(T)=T^6+1\,.$}\\\hline

$(7,3)$  & $\{1,3,5\}$ & $\{1\}$ &  $  3 $ &  $\{\textbf 3,2,6\,|\,4,\textbf 5,\textbf 1\}$  \\\hline
\multicolumn{5}{|c|}{
$
\gamma=
\begin{pmatrix}
0 & J_2 & 0\\
0 & 0 & I_2\\
I_2 & 0 & 0
\end{pmatrix}\qquad
U=
\begin{pmatrix}
U_1 & 0 & 0\\
0 & U_2 & 0\\
0 & 0 & U_3
\end{pmatrix}
$
}\\[6pt]
\multicolumn{5}{|c|}{$P_{\gamma^3}(T)=(T^2+1)^3,\quad P_{\gamma^2}(T)=T^6+(u_1u_2\overline u_3+\overline u_1\overline u_2 u_3)T^3+1,\quad P_{\gamma}(T)=T^6+1\,.$}\\\hline

$(11,1)$  & $\{1,2,3,4,5\}$  & $\{1\}$  &  $ 2$ & $\{\textbf 2,\textbf 4,8,\textbf 5,10\,|\,9,7,\textbf 3,6,\textbf 1\}$  \\\hline
\multicolumn{5}{|c|}{
$
\gamma=
\begin{pmatrix}
0 & I_2 & 0 & 0 & 0\\
0 & 0 & J_2 & 0 & 0\\
0 & 0 & 0 & J_2 & 0\\
0 & 0 & 0 & 0 & J_2\\
I_2 & 0 & 0 & 0 & 0
\end{pmatrix}\qquad
U=
\begin{pmatrix}
U_1 & 0 & 0 & 0 & 0\\
0 & U_2 & 0 & 0 & 0\\
0 & 0 & U_3 & 0 & 0\\
0 & 0 & 0 & U_4 & 0\\
0 & 0 & 0 & 0 & U_5
\end{pmatrix}
$}\\[6pt]
\multicolumn{5}{|c|}{$P_{\gamma^5}(T)=(T^2+1)^5,\quad P_{\gamma^2}(T)=T^{10}+(u_1\overline u_2\overline u_3\overline u_4\overline u_5+\overline u_1 u_2 u_3 u_4 u_5)T^5+1,\quad P_{\gamma}(T)=T^{10}+1\,.$}\\\hline

$(13,2)$  & $\{1,2,3,4,7,8\}$ & $\{  1 \}$  &  $2$ & $\{\textbf 2,\textbf 4,\textbf 8,\textbf 3,6,12\,|\,11,9,5,10,\textbf 7,\textbf 1\}$  \\\hline
\multicolumn{5}{|c|}{
$
\gamma=
\begin{pmatrix}
0 & I_2 & 0 & 0 & 0 & 0\\
0 & 0 & I_2 & 0 & 0 & 0\\
0 & 0 & 0 & I_2 & 0 & 0\\
0 & 0 & 0 & 0 & J_2 & 0\\
0 & 0 & 0 & 0 & 0 & I_2\\
I_2 & 0 & 0 & 0 & 0 & 0
\end{pmatrix}\qquad
U=
\begin{pmatrix}
U_1 & 0 & 0 & 0 & 0 & 0\\
0 & U_2 & 0 & 0 & 0 & 0\\
0 & 0 & U_3 & 0 & 0 & 0\\
0 & 0 & 0 & U_4 & 0 & 0\\
0 & 0 & 0 & 0 & U_5 & 0\\
0 & 0 & 0 & 0 & 0 & U_6\\
\end{pmatrix}
$}\\[6pt]
\multicolumn{5}{|c|}{$P_{\gamma^6}(T)=(T^2+1)^6,\quad P_{\gamma^4}(T)=(T^{6}+(u_1\overline u_3\overline u_5+\overline u_1 u_3u_5)T^3+1)(T^{6}+(u_2\overline u_4\overline u_6+\overline u_2 u_4u_6)T^3+1),$}\\
\multicolumn{5}{|c|}{$P_{\gamma^3}(T)=(T^4+1)^3,\quad P_{\gamma^2}(T)=(T^{6}+1)^2,\quad P_{\gamma}(T)=T^{12}+1$}\\\hline

$(13,3)$  & $\{1,2,3,5,6,9\}$ & $\{  1,3,9 \}$  &  $2$ & $\{\textbf 2,4, 8,\textbf 3,\textbf 6,12\,|\,11,\textbf 9,\textbf 5,10, 7,\textbf 1\}$  \\\hline
\multicolumn{5}{|c|}{
$
\gamma=
\begin{pmatrix}
0 & J_2 & 0 & 0 & 0 & 0\\
0 & 0 & I_2 & 0 & 0 & 0\\
0 & 0 & 0 & J_2 & 0 & 0\\
0 & 0 & 0 & 0 & I_2 & 0\\
0 & 0 & 0 & 0 & 0 & J_2\\
I_2 & 0 & 0 & 0 & 0 & 0
\end{pmatrix}\qquad
U=
\begin{pmatrix}
U_1 & 0 & 0 & 0 & 0 & 0\\
0 & U_2 & 0 & 0 & 0 & 0\\
0 & 0 & U_1 & 0 & 0 & 0\\
0 & 0 & 0 & U_2 & 0 & 0\\
0 & 0 & 0 & 0 & U_1 & 0\\
0 & 0 & 0 & 0 & 0 & U_2\\
\end{pmatrix}
$}\\[6pt]
\multicolumn{5}{|c|}{$P_{\gamma^6}(T)=(T^2+1)^6,\quad P_{\gamma^4}(T)=(T^{6}+(u_1^3+\overline u_1 ^3)T^3+1)(T^{6}+(u_2^3+\overline u_2 ^3)T^3+1),$}\\
\multicolumn{5}{|c|}{$P_{\gamma^3}(T)=(T^4+1)^3,\quad P_{\gamma^2}(T)=(T^{6}+1)^2,\quad P_{\gamma}(T)=T^{12}+1$}\\\hline

\end{tabular}
\end{center}
\end{table}

\begin{table}
\begin{center}
\caption{First moments of the measures $\mu_i$. For $i$ and $n$ odd, $\M_n[\mu_i]=0$, and we do not write these moments on the table.}\label{table: moments}
\vspace{5pt}
\begin{tabular}{|r|r|r|r|r|}\hline
$(\ell,k)$ & $\M_2[\mu_1]$ & $\M_4[\mu_1]$ & $\M_6[\mu_1]$ & $\M_8[\mu_1]$  \\\hline
$(5,2)$ & 1 & 9 & 100 & 1225  \\
$(7,2)$ & 3 & 81 & 2430 & 76545  \\
$(7,3)$ & 1 & 15 & 310 & 7455  \\
$(11,1)$ & 1 & 27 & 1090 & 55195  \\
$(13,2)$ & 1 & 33 & 1660 & 106785  \\
$(13,3)$ & 3 & 243 & 24300 & 2679075  \\\hline
$(\ell,k)$ & $\M_1[\mu_2]$ & $\M_2[\mu_2]$ & $\M_3[\mu_2]$ & $\M_4[\mu_2]$  \\\hline
$(5,2)$ &  1 & 3 & 10 & 41  \\
$(7,2)$ &  2 & 18 & 207 & 2646 \\
$(7,3)$ & 1 & 5 & 35 & 321  \\
$(11,1)$ &   1 & 9 & 133 & 2873  \\
$(13,2)$ &  1  & 11 & 206 & 5781  \\
$(13,3)$ &  2 & 60 & 2610 & 130842  \\\hline
$(\ell,k)$ & $\M_2[\mu_3]$ & $\M_4[\mu_3]$ & $\M_6[\mu_3]$ & $\M_8[\mu_3]$  \\\hline   
$(7,2)$ &  28 & 7860 & 2575810 & 893661020  \\ 
$(7,3)$ &  6 & 822 & 184860 & 48884710 \\
$(11,1)$ & 24 & 73176 & 406662720 & 2941907232600 \\
$(13,2)$ &  39 & 287391 & 4433856900 & 93962238664175 \\
$(13,3)$ &  487 & 12209463 & 398722297600 & 14560811533839655  \\\hline
$(\ell,k)$ & $\M_1[\mu_4]$ & $\M_2[\mu_4]$ & $\M_3[\mu_4]$ & $\M_4[\mu_4]$   \\\hline
$(11,1)$ &  1 & 54 & 4588 & 497236  \\
$(13,2)$ &  3 & 139 & 20267 & 4480911 \\
$(13,3)$ &  10 & 2142 & 712107 & 266575698\\\hline
$(\ell,k)$ & $\M_2[\mu_5]$ & $\M_4[\mu_5]$ & $\M_6[\mu_5]$ & $\M_8[\mu_5]$  \\\hline
$(11,1)$ &  72 & 934332 & 22782049800 & 725020102732940 \\
$(13,2)$ &  236 & 22587768 & 4493470904960 & 1230243879356591400  \\
$(13,3)$ & 5004 & 1604318076 & 675819691911360 &
319107416394892272084  \\\hline
$(\ell,k)$ & $\M_1[\mu_6]$ & $\M_2[\mu_6]$ & $\M_3[\mu_6]$ & $\M_4[\mu_6]$  \\\hline
$(13,2)$ &  4 & 334 & 93100 & 38562182  \\
$(13,3)$ &  16 & 6678 & 4147390 & 2893450202  \\\hline
\end{tabular}
\end{center}
\end{table}

\begin{table}
\begin{center}
\caption{Some moment statistics of normalized traces for $x=2^{27}$. }\label{table: statistical moments}
\vspace{5pt}
\begin{tabular}{|r|r|r|r|r|r|r|r|r|r|}\hline
$(\ell,k)$ & $M_{1,x}$ & $M_{2,x}$ & $M_{3,x}$ & $M_{4,x}$ & $M_{5,x}$ & $M_{6,x}$  & $M_{7,x}$ & $M_{8,x}$ \\\hline
$(5,2)$ &   -0.000 & 1.010  & -0.002 & 9.084 & -0.030 & 100.877 & -0.366 & 1235.171 \\
$(7,2)$ & -0.000 & 2.999  & -0.000 & 80.984 & -0.009 & 2429.414 & 0.674 & 76523.229  \\
$(7,3)$ & 0.000 & 0.999 & 0.000 & 14.979 & 0.011 & 309.265 & 0.722 & 7428.375  \\
$(11,1)$ & -0.000 & 0.999 & -0.007 & 26.907 & -0.203 & 1080.500 & -3.930 & 54274.737  \\
$(13,2)$ & -0.000 & 1.001 & 0.004 & 32.948 & 0.376 & 1646.380 & 43.571 & 104860.429  \\
$(13,3)$ & -0.000 & 3.002 & -0.026 & 243.122 & -2.262 & 24306.084 & -199.309 & 2679022.039  \\\hline
\end{tabular}
\end{center}
\end{table}

\end{document}